\documentclass[12pt]{article}

\usepackage{amssymb}
\usepackage{amsthm}
\usepackage{amsmath}
\usepackage{graphicx}
\usepackage{fullpage}
\usepackage{subcaption}

\usepackage{color}
\usepackage{enumerate}
 \numberwithin{equation}{section}
\usepackage{booktabs}
\allowdisplaybreaks

\usepackage[pdftex]{hyperref}
 \usepackage{mathtools}
 \mathtoolsset{showonlyrefs}
 \usepackage[nocompress]{cite}

\theoremstyle{plain}
\newtheorem{thm}{Theorem}[section]
\newtheorem{cor}[thm]{Corollary}

\newtheorem{lem}[thm]{Lemma}
\newtheorem{prop}[thm]{Proposition}

\theoremstyle{definition}

\newtheorem{ex}[thm]{Example}

\theoremstyle{remark}

\newcommand{\R}{\mathbb{R}}

\newcommand{\bp}{\begin{proof}[\ensuremath{\mathbf{Proof}}]}
\newcommand{\bs}{\begin{proof}[\ensuremath{\mathbf{Solution}}]}
\newcommand{\ep}{\end{proof}}
\newcommand{\be}{\begin{equation}}
\newcommand{\ee}{\end{equation}}

\begin{document}

\title{The perimeter and volume of a Reuleaux polyhedron}

\author{Ryan Hynd\footnote{Department of Mathematics, University of Pennsylvania. Supported in part by an American Mathematics Society Claytor--Gilmer Fellowship.}}

\maketitle

\begin{abstract}
A ball polyhedron is the intersection of a finite number of closed balls in $\R^3$ with the same radius. In this note, we study ball polyhedra in which the set of centers defining the balls have the maximum possible number of diametric pairs. We explain how to compute the perimeter and volume of these shapes by employing the Gauss--Bonnet theorem and another integral formula.  In addition, we show how to adapt this method to approximate the volume of Meissner polyhedra, which are constant width bodies constructed from ball polyhedra. 
\end{abstract}

\tableofcontents

\section{Introduction}
Let us first consider the vertices $\{x_1,x_2,x_3,x_4\}$ of a regular tetrahedron in $\R^3$ of side length one. That is, $\{x_1,x_2,x_3,x_4\}$ are four points in $\R^3$ which satisfy
$$
|x_i-x_j|=1\; \text{for all $i\neq j$}. 
$$
The corresponding {\it Reuleaux tetrahedron} $R$ is given by
$$
R=B(x_1)\cap B(x_2)\cap B(x_3)\cap B(x_4).
$$
Here $B(x)\subset \R^3$ denotes the closed ball of radius one centered at $x$.  A basic problem is to determine the perimeter and volume of $R$. 
\begin{figure}[h]
\centering
 \includegraphics[width=.45\textwidth]{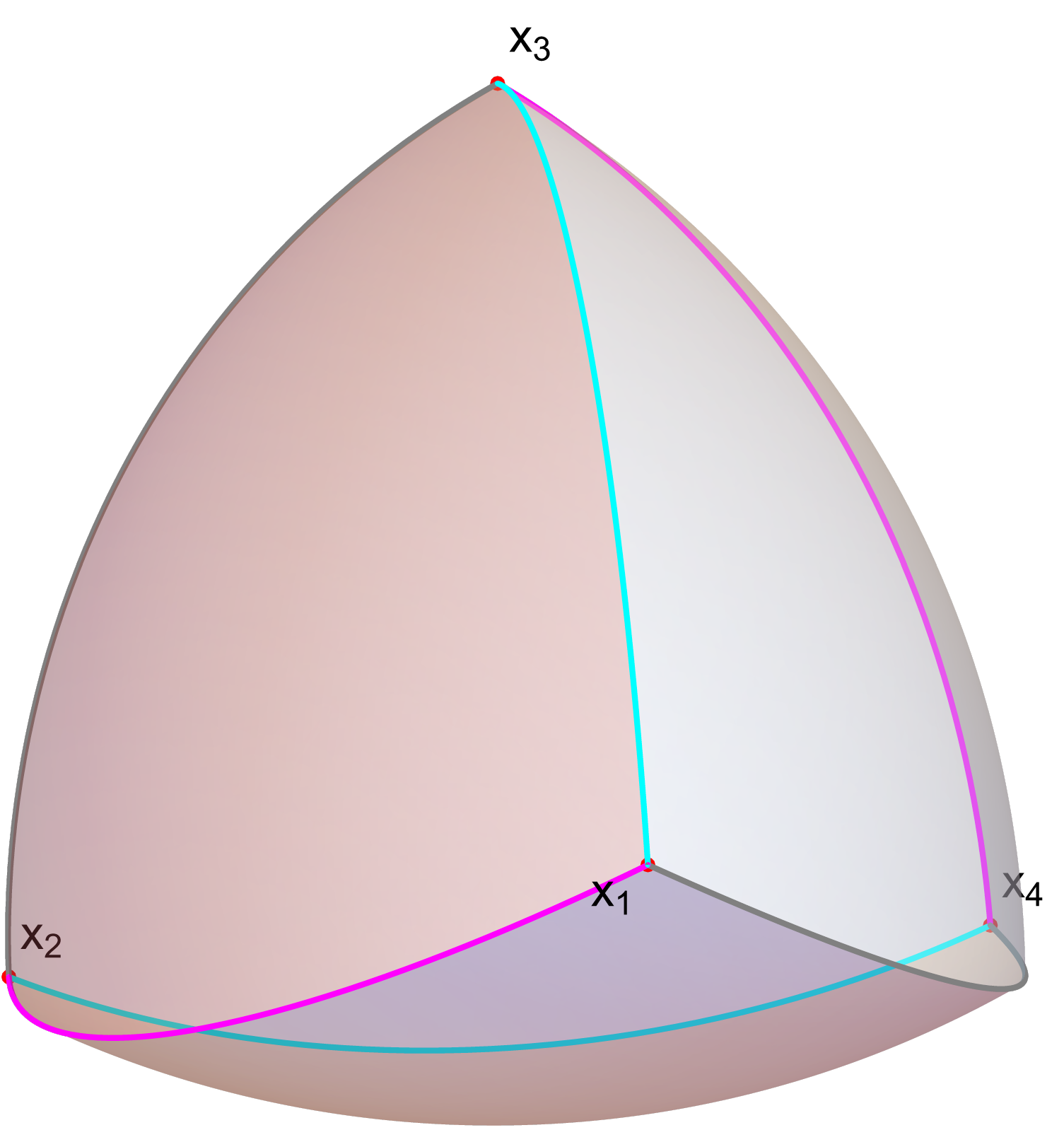}
 \hspace{.2in}
  \includegraphics[width=.43\textwidth]{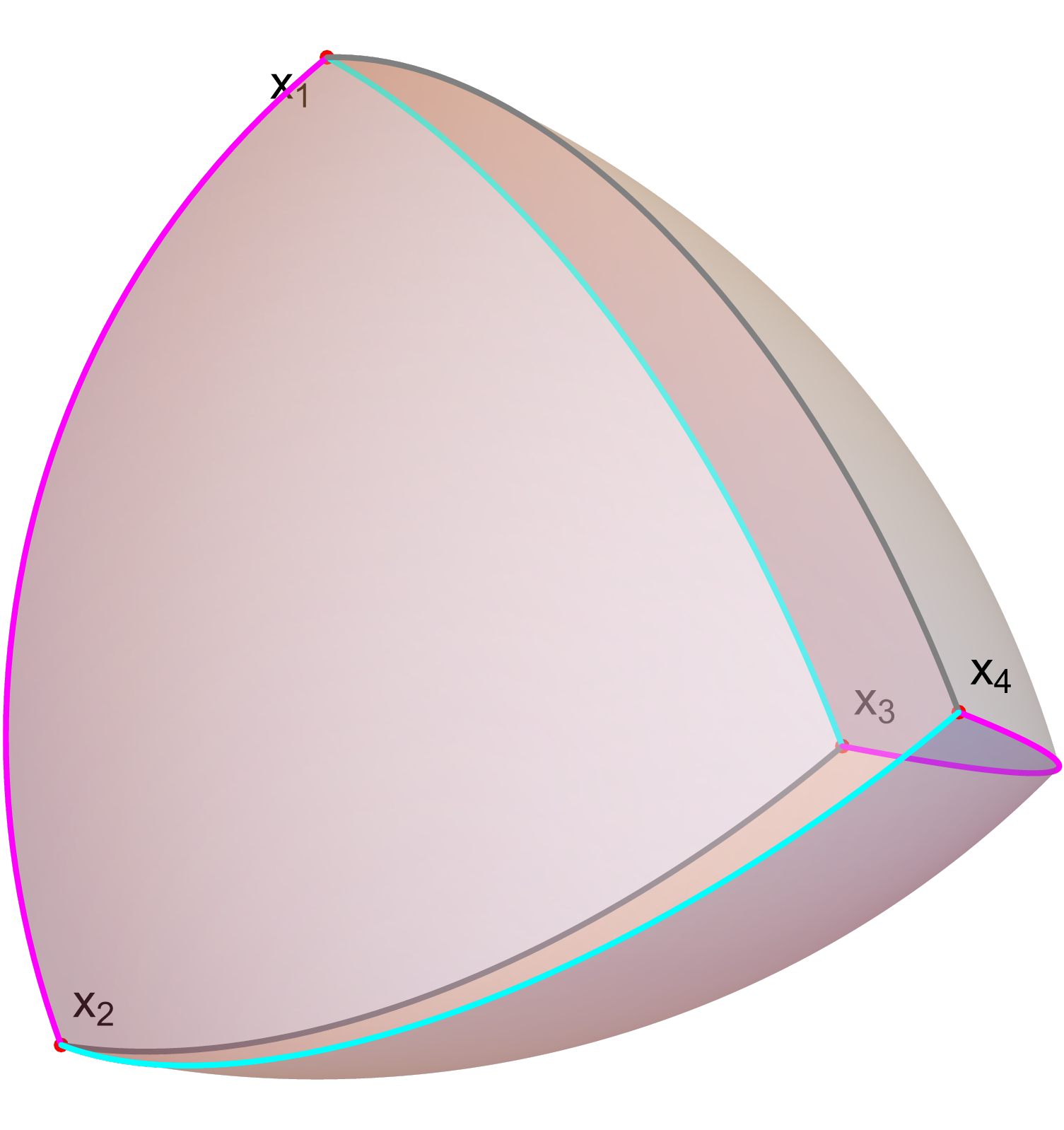}
 \caption{Here are two translucent views of a Reuleaux tetrahedron $R$ defined by the centers $\{x_1,x_2,x_3,x_4\}$.   Observe that the centers $\{x_1,x_2, x_3,x_4\}$ are also the vertices of $R$.  Moreover, $R$ has four vertices, six edges, and four faces just like a regular tetrahedron in $\R^3$.  }\label{ReuleauxTetra}
\end{figure}

\par As explained by Harbourne \cite{harbourne}, we can use the Gauss--Bonnet theorem to compute the surface area of each face of $R$ as $2\pi-(9/2)\cos^{-1}(1/3)$.  Therefore, 
the perimeter of $R$ is 
$$
P(R)=4\left(2\pi-\frac{9}{2}\cos^{-1}(1/3)\right).
$$
Harbourne also explained how to compute the volume
$$
V(R)=\frac{8\pi}{3}+\frac{\sqrt{2}}{4}-\frac{27}{4}\cos^{-1}(1/3)
$$
by exploiting the symmetry of $R$.  In this note, we extend these ideas to a family of shapes which include Reuleaux tetrahedra. 

\subsection{Reuleaux polyhedra} Suppose that $X\subset \R^3$ is finite with at least four points and that $X$ has diameter one.  We'll say that a pair $\{x,y\}\subset X$ is {\it diametric} if $|x-y|=1$. We'll also say that $X$ is {\it extremal} provided that $X$ has the maximal possible number of diametric pairs.  Gr\"unbaum \cite{MR87115}, Heppes \cite{MR87116}, and Straszewicz \cite{MR0087117} independently verified the conjecture of V\'azsonyi which asserted that 
$$
\text{if $X$ has $m$ elements, then $X$ has $\le 2m-2$ diametric pairs.}
$$
Therefore, $X$ is extremal if and only if it has $2m-2$ diametric pairs.  

\par Gr\"unbaum, Heppes, and Straszewicz each employed the convex body 
$$
B(X):=\bigcap_{x\in X}B(x)
$$
in their respective works \cite{MR87115,MR87116,MR0087117}. In general, an intersection of finitely many congruent balls is known as a ball polyhedron. When $X$ is extremal we say that $B(X)$ is a {\it Reuleaux polyhedron}. Our goal is to describe an efficient way to compute the perimeter and volume of Reuleaux polyhedra.

\par  Let us suppose for definiteness that 
$$
\text{$X=\{x_1,\dots,x_m\}\subset \R^3$ is extremal.}
$$
In order to state our formulae for the perimeter and volume of $B(X)$, we will need to recall some key properties of the boundary of $B(X)$. The following results regarding the boundary structure of $B(X)$ have been verified in the seminal work by Kupitz, Perles, and Martini \cite{MR2593321} on extremal subsets of $\R^3$.  We also recommend 
related work by Bezdek,  L\'angi, Nasz\'odi, and Papez  \cite{MR2343304} which discusses various important properties of ball polyhedra.
\\\\
\noindent {\bf Faces}. As $B(X)=B(x_1)\cap\dots\cap B(x_m)$, 
\be\label{BoundaryUnionFormula}
\partial B(X)=\bigcup^m_{j=1}\partial B(x_j)\cap B(X).
\ee
Here $\partial B(x_j)\cap B(X)$ is the {\it face} of $B(X)$ which is opposite $x_j$.  Furthermore, each face of $B(X)$ is distinct, so $B(X)$ has a face opposite to each center in $X$. It is also known that $\partial B(x_j)\cap B(X)$ is a geodesically convex subset of $ \partial B(x_j)$.
\\\\
\noindent {\bf Vertices}.   A {\it principal vertex} of $B(X)$ is a point which belongs to three or more faces of $B(X)$, and a {\it dangling vertex}  is a member of $X$ which belongs to exactly two faces of $B(X)$.  The set $X$ being extremal is equivalent to 
$$
X=\text{vert}(B(X)),
$$
where $\text{vert}(B(X))$ denotes the collection of all vertices (principal and dangling) of $B(X)$.  Therefore, each center defining $B(X)$ is a vertex of $B(X)$ and conversely. 
\\\\
\noindent {\bf Edges}. The {\it edges} of $B(X)$ are the connected components of 
$$
(\partial B(x_j)\cap \partial B(x_k)\cap B(X))\setminus X\quad (j\neq k).
$$
\begin{figure}[h]
\centering
 \includegraphics[width=.45\textwidth]{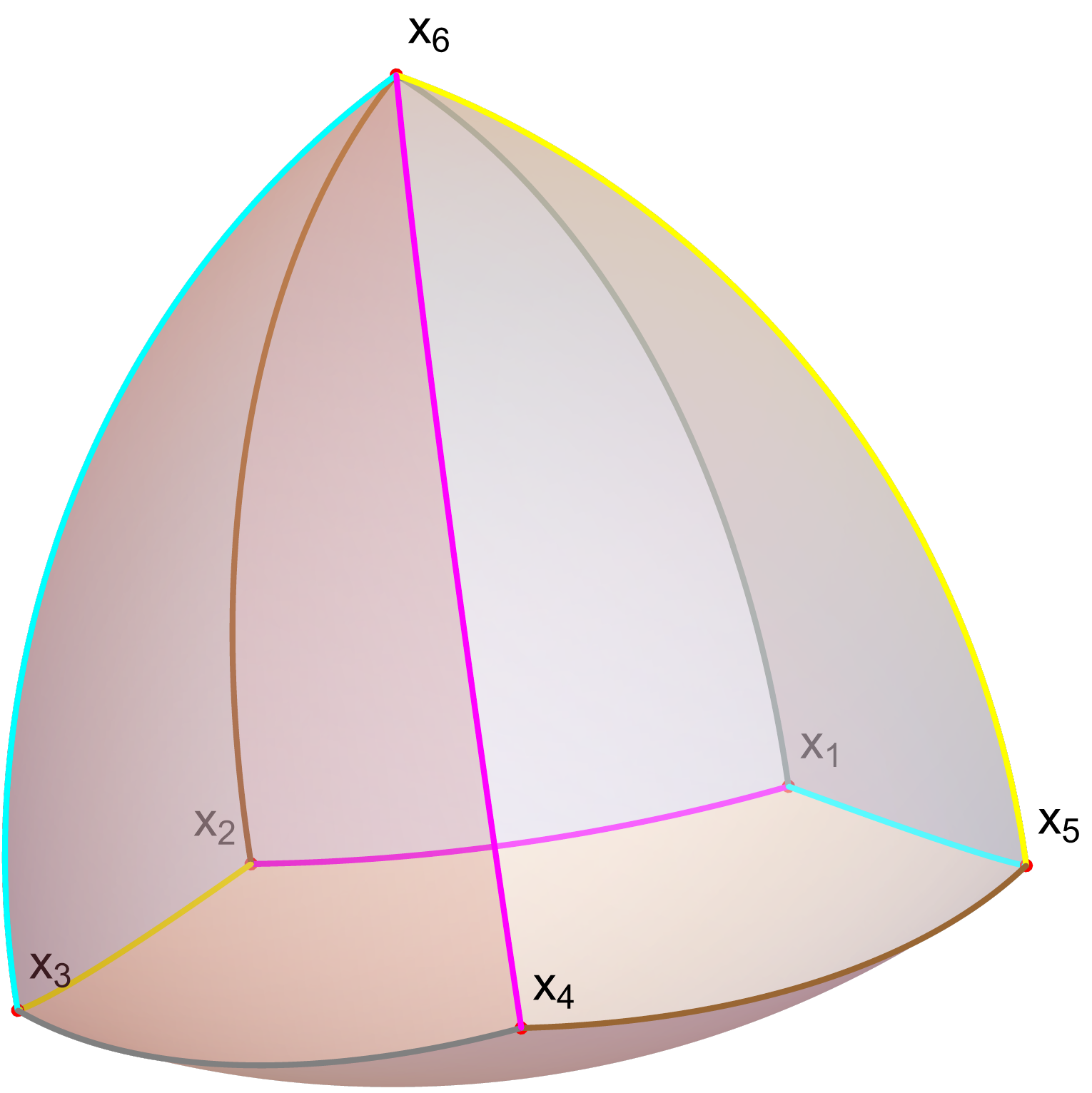}
 \hspace{.2in}
  \includegraphics[width=.45\textwidth]{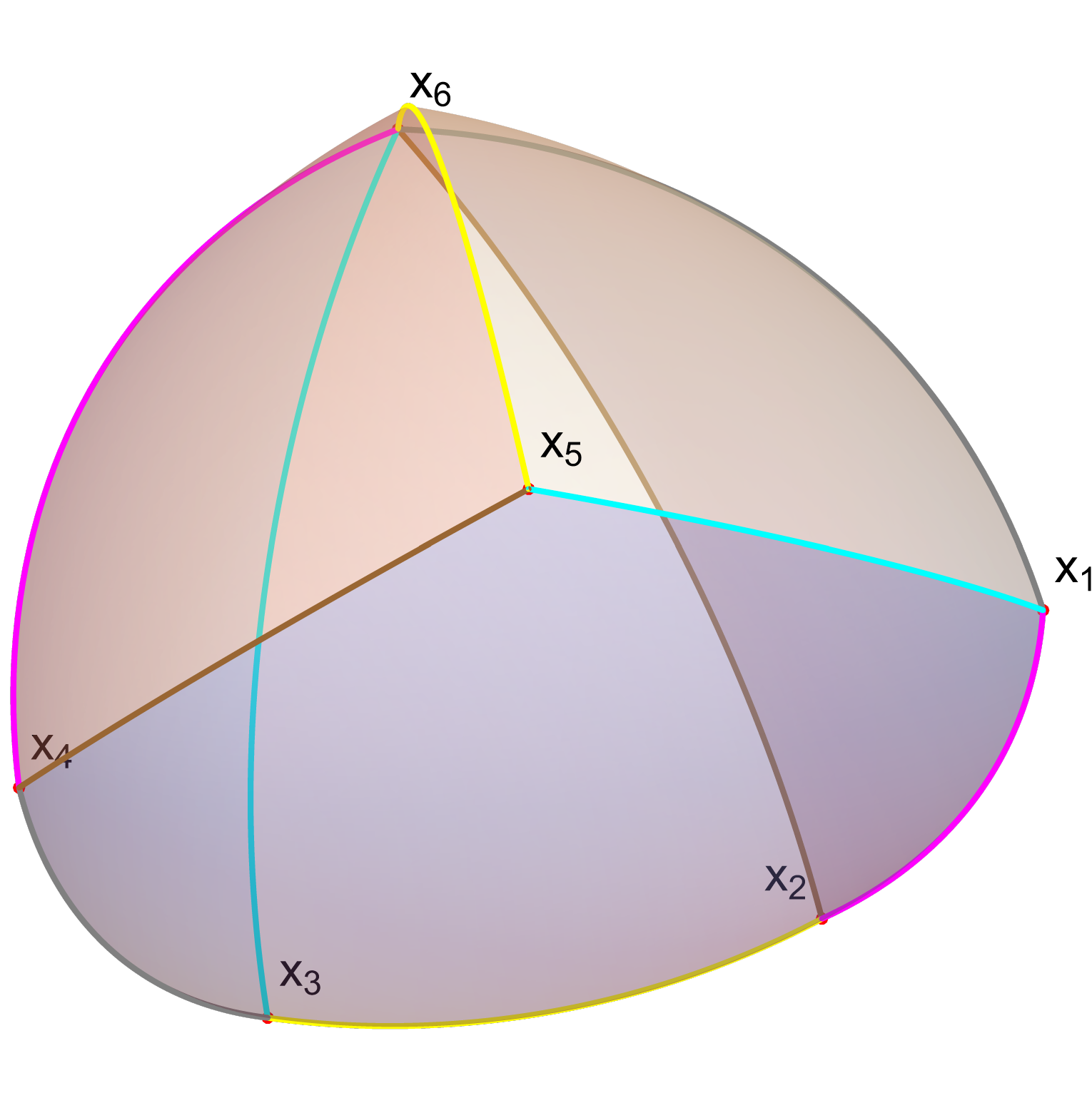}
 \caption{A Reuleaux polyhedron $B(\{x_1,\dots,x_6\})$ displayed with its dual edge pairs.  Note in particular that we have drawn both edges in a dual edge pair with the same color. }\label{ReulauxTetraEdge}
\end{figure}
In particular, each edge is a segment of a circular arc since $\partial B(x_j)\cap \partial B(x_k)$ is the circle defined by the equations for $x\in \R^3 $
\be\label{EdgeCircle}
\displaystyle \left|x-\frac{x_j+x_k}{2}\right|=\sqrt{1-\left|\frac{x_j-x_k}{2}\right|^2}\;\text{and}\; \displaystyle \left(x-\frac{x_j+x_k}{2}\right)\cdot \left(x_j-x_k\right)=0.
\ee
It turns out that these edges are naturally grouped in pairs.  That is, for each edge 
$$
e\subset \partial B(x_j)\cap \partial B(x_k)\cap B(X)
$$
with endpoints $x_{j'}$ and $x_{k'}$, there is a unique edge 
$$
e'\subset \partial B(x_{j'})\cap \partial B(x_{k'})\cap B(X)
$$
with endpoints $x_j$ and $x_k$. In this case, $e$ and $e'$ are {\it dual edges}.
\\\\
\noindent {\bf Positively oriented opposite vertices}. We will need to introduce some notation in order to state our perimeter and volume formulae below concisely. To this end, we fix $j\in \{1,\dots, m\}$ and choose 
\be\label{aeyejay}
\begin{array}{p{0.8\textwidth}}
vertices $a_{1,j},\dots, a_{N_j,j}$ of $B(X)$ which belong to the face opposite $x_j$ such 
that $a_{i,j}$ and $a_{i+1,j}$ are joined by an edge $e_{i,j}$ of $B(X)$ for $i=1,\dots, N_j$.
\end{array}
\ee
Here $ a_{N_j+1,j}=a_{1,j}$, so $e_{N_j,j}$ joins $a_{N_j,j}$ to $a_{1,j}$.  We will say that $\{a_{1,j},\dots, a_{N_j,j}\}$ is {\it positively oriented} in $\partial B(x_j)$ if the interior of the face opposite $x_j$ is on the left as one successively traverses the edges $e_{1,j},e_{2,j},\dots, e_{N_j,j}$ from the exterior of $ B(x_j)$.   In addition, we note the edge dual to $e_{ij}$ has endpoints $x_j$ and $b_{i,j}\in X$ for $i=1,\dots, N_j.$ Therefore, 
\be\label{beyejay}
\text{ $e_{ij}\subset \partial B(x_j)\cap \partial B(b_{i,j})\cap B(X)$ for $i=1,\dots, N_j.$}
\ee
Relative to $\{a_{1,j},\dots, a_{N_j,j}\}$, we will say $\{b_{1,j},\dots, b_{N_j,j}\}$ are the {\it corresponding adjacent vertices} of $B(X)$.  See Figure \ref{EdgeDiagram} for a related diagram. 
\begin{figure}[h]
\centering
 \includegraphics[width=.6\textwidth]{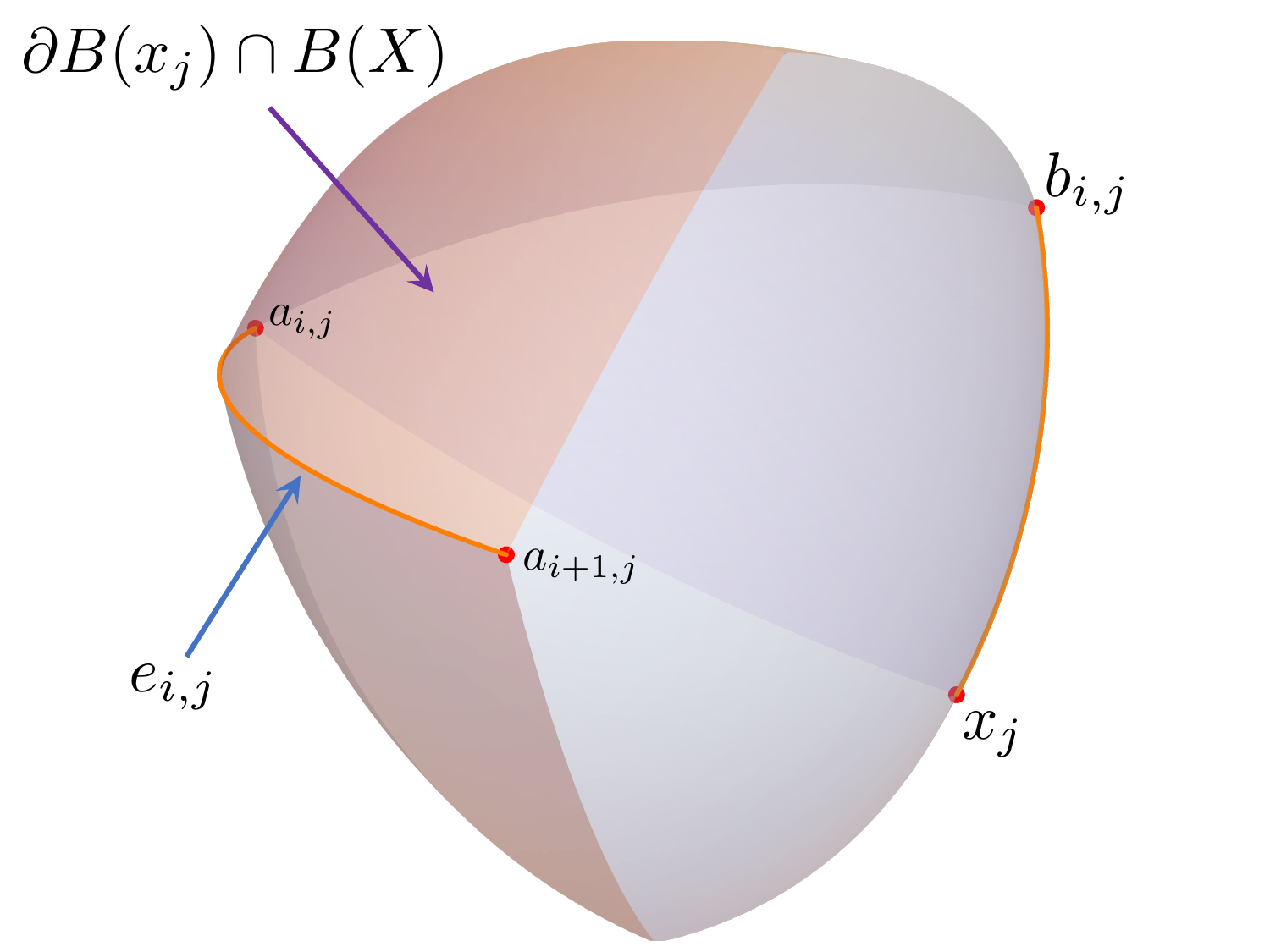}
 \caption{This diagram was made to give some intuition about the edge $e_{i,j}$ in the statements \eqref{aeyejay} and \eqref{beyejay}. Note that $e_{i,j}$ belongs to the face opposite $x_j\in X$ and joins two vertices $a_{i,j}$ and $a_{i+1,j}$. Also observe that the edge dual to $e_{i,j}$ joins $b_{i,j}$ and $x_j$.  } \label{EdgeDiagram}
\end{figure}
\\
\par Let us see how this works when $X=\{x_1,x_2,x_3,x_4\}$ is the set of vertices of a regular tetrahedron in  $\R^3$ of side length one. Note that there are six diametric pairs  $\{x_i,x_j\}$ with $1\le i<j\le 4$. 
As $6=2\cdot 4-2$, $X$ is extremal. Moreover, $R=B(\{x_1,x_2,x_3,x_4\})$ has four vertices $\{x_1,x_2,x_3,x_4\}$, six edges 
$\partial B(x_i)\cap \partial B(x_j)\cap R$,  four faces $\partial B(x_i)\cap R$ $(i,j=1,\dots,4)$, and 3 dual edge pairs 
$$
\partial B(x_i)\cap \partial B(x_j)\cap R\;\text{and}\;\partial B(x_k)\cap \partial B(x_\ell)\cap R,
$$
where $i,j,k,\ell\in \{1,2,3,4\}$ are distinct. In reference to $R$ in Figure \ref{ReuleauxTetra}, 
we may choose a collection of positively oriented vertices opposite to $x_1$ with corresponding adjacent vertices  
$$
\begin{cases}
a_{1,1}=x_2\\
a_{2,1}=x_3\\
a_{3,1}=x_4
\end{cases}
\quad 
\begin{cases}
b_{1,1}=x_4\\
b_{2,1}=x_2\\
b_{3,1}=x_3
\end{cases}
$$
and likewise for $x_2, x_3, x_4$.

\subsection{Main results}
As mentioned, we will show how to extend the method Harbourne used to compute the perimeter of $R$ to all Reuleaux polyhedra $B(X)$. This method is based on an application of the Gauss--Bonnet theorem.  Let us recall that if $F\subset\partial B(0)$ and $\partial F$ is a simple, closed, positively oriented curve which is the union of smooth curves $\gamma_1,\dots, \gamma_N$, then
$$
\int_{F}K d\sigma+\sum^N_{j=1}\left(\int_{\gamma_j}k_gds+\theta_j\right)=2\pi.
$$
Here $K$ is the Gaussian curvature of $\partial B(0)$ in the region $F$, $k_g$ in the integral $\int_{\gamma_j}k_gds$ is the geodesic curvature of $\gamma_j$, and $\theta_j$ is the angle between the slopes of $\gamma_j$ and $\gamma_{j+1}$ where these curves meet.  We have used $\sigma$ to denote the corresponding surface measure on $\partial B(0)$. 

\par The key observation made by Harbourne is that $K\equiv 1$ on $\partial B(0)$, so 
$$
\sigma(F)=2\pi-\sum^N_{j=1}\left(\int_{\gamma_j}k_gds+\theta_j\right).
$$
Our task is then to compute the integrals of the geodesic curvature and the external angles to get an explicit formula for the surface area of each face of a Reuleaux polyhedron $B(X)$. We can do this using the identity $X=\text{vert}(B(X))$ and edge duality. In particular, these properties allow us to describe the boundary curves of each face of $B(X)$ solely in terms of $X$. This is not possible for every ball polyhedra; one would typically have to find outstanding principal vertices before proceeding to compute the surface area of a given face. See Figure \ref{BallPoly} for an example.  
\begin{figure}[h]
\centering
 \includegraphics[width=.45\textwidth]{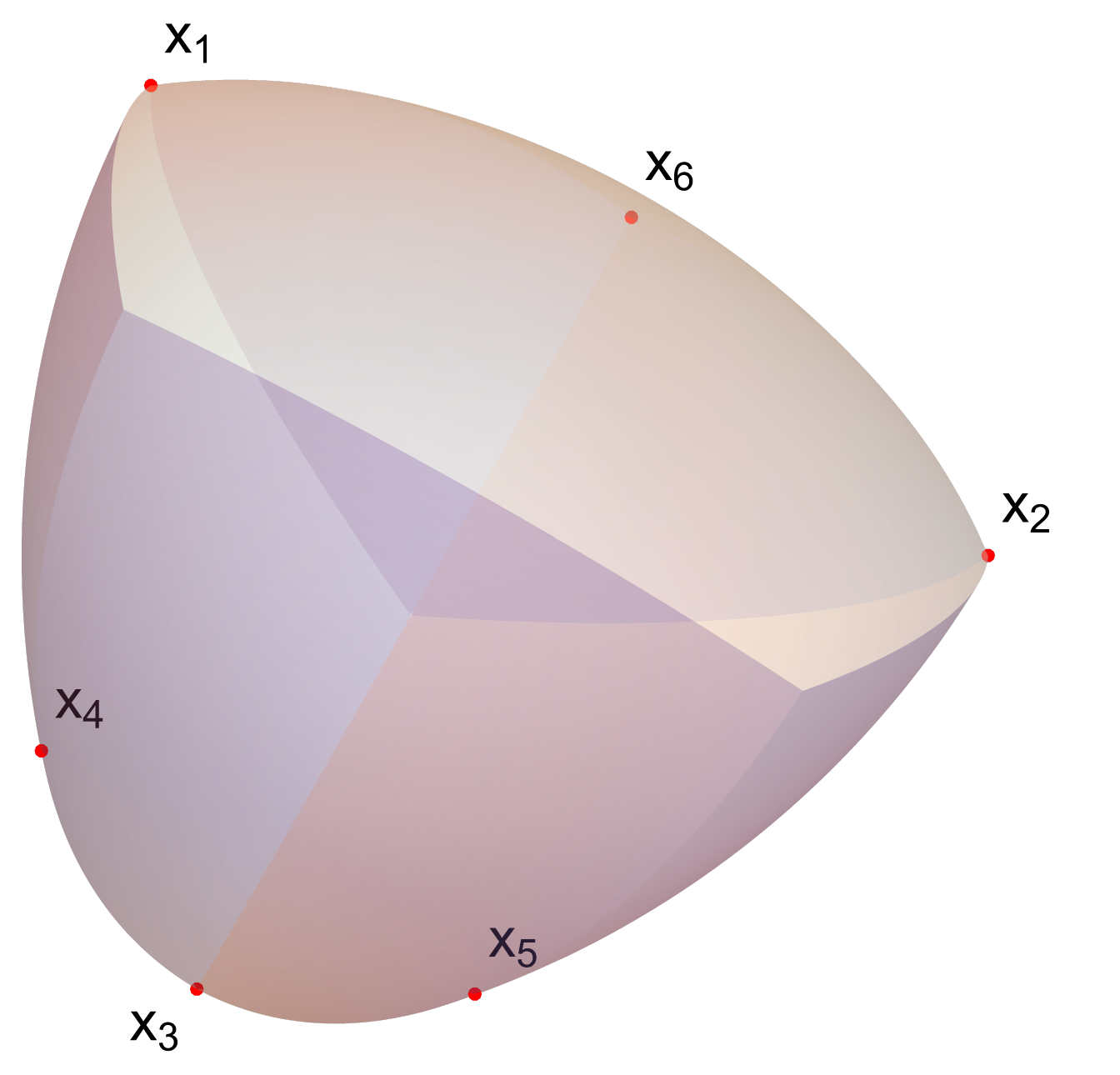}
 \hspace{.2in}
  \includegraphics[width=.45\textwidth]{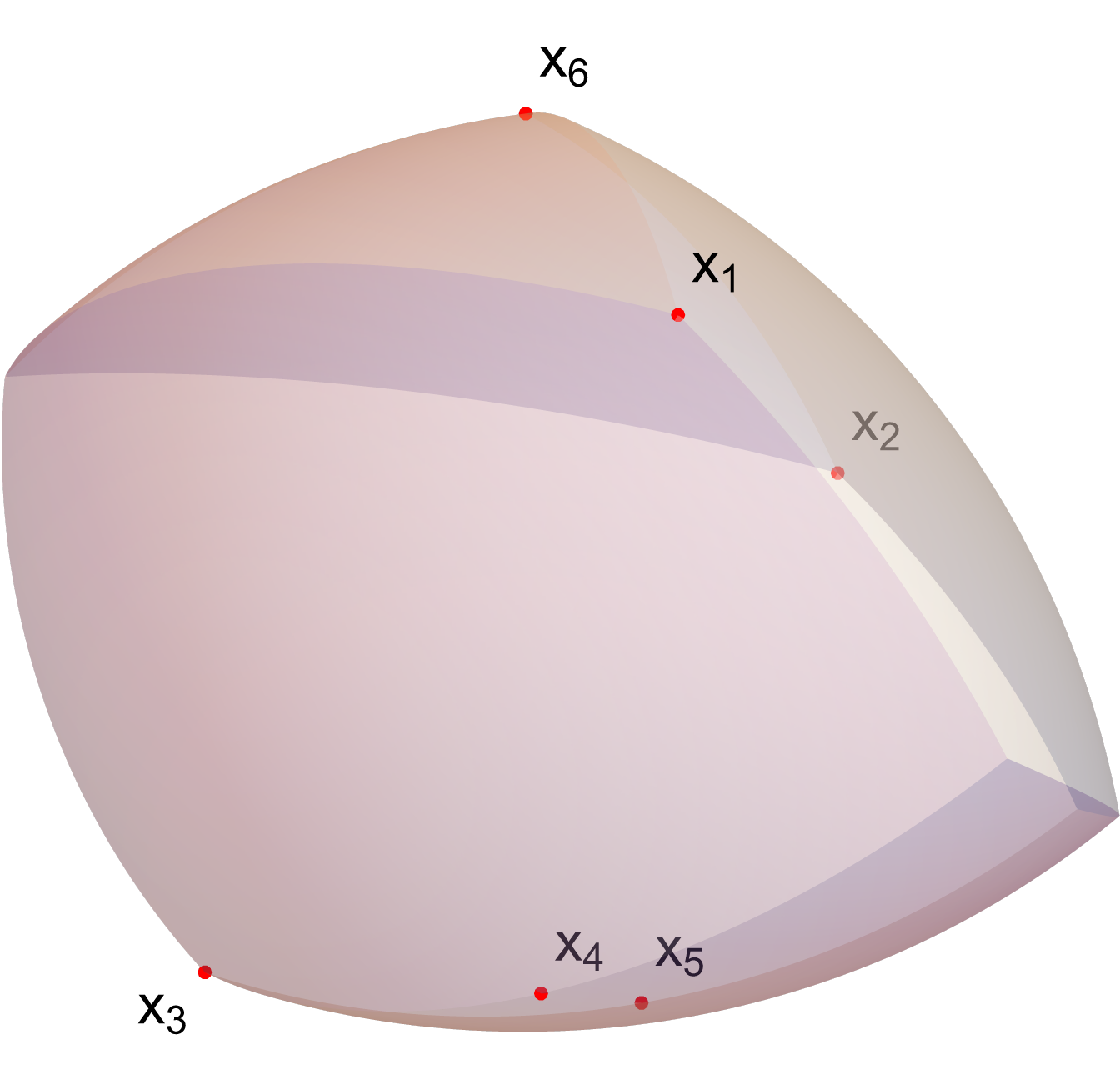}
 \caption{This is an example of a ball polyhedra $B(X)$ such that each vertex of $B(X)$ does not belong to the set of centers $X$.  Therefore, in order to use the Gauss--Bonnet theorem to compute the perimeter of $B(X)$, one would need to find the outstanding vertices to parametrize the edges of each face. This extra step is not necessary for a Reuleaux polyhedra since the set of centers and vertices coincide. }\label{BallPoly}
\end{figure}

\par In the statement below, we will use the notation 
$$
\angle(a,b):=\cos^{-1}\left(\frac{a}{|a|}\cdot \frac{b}{|b|}\right)
$$
for the angle between $a,b\in \R^3$ with $|a|,|b|\neq 0$. 
\begin{thm}\label{PerThm}
Suppose $X=\{x_1,\dots, x_m\}\subset \R^3$ is extremal. Choose positively oriented vertices $\{a_{i,j}\}_{j=1}^{N_j}$ opposite to $x_j$ with corresponding adjacent vertices $\{b_{i,j}\}_{j=1}^{N_j}$ for $j=1,\dots, m$. 
The perimeter of $B(X)$ is given by 
$$
P(B(X))=\sum^m_{j=1}\left\{2\pi -\sum^{N_j}_{i=1}\left(\left|\frac{b_{i,j}-x_j}{2}\right|\psi_{i,j}+\theta_{i,j}\right)\right\},
$$
where 
$$
\psi_{i,j}=\angle\left(a_{i,j}-\frac{b_{i,j}+x_j}{2},a_{i+1,j}-\frac{b_{i,j}+x_j}{2}\right)
$$
and 
$$
\theta_{i,j}=\angle\left((b_{i,j}-x_j)\times \left(a_{i+1,j}-\frac{b_{i,j}+x_j}{2}\right),(b_{i+1,j}-x_j)\times \left(a_{i+1,j}-\frac{b_{i+1,j}+x_j}{2}\right)\right).
$$
\end{thm}
\par Next we will compute the volume of a Reuleaux polyhedron $B(X)$. To this end, we will first express the volume of $B(X)$ as an integral over the boundary 
via the divergence theorem
\be\label{divThmFormula}
V(B(X))=\frac{1}{3}\int_{ B(X)} \nabla \cdot x\; dx=\frac{1}{3}\int_{\partial B(X)}x\cdot N d\sigma
\ee
Here the surface integral is taken over the $m$ faces of $\partial B(X)$ and $\sigma$ denotes the corresponding surface measure on each face with outward normal $N$.  After performing an integration by parts, we will arrive at the following expression. 
\begin{thm}\label{VolThm}
With the same hypotheses and notation as Theorem \ref{PerThm}, the volume of $B(X)$ is given by 
\begin{align}
&V(B(X))=\frac{1}{3}P(B(X)) \;\; +\\
&\quad \quad \frac{1}{6}\sum^m_{j=1}x_j\cdot\sum^{N_j}_{i=1}\left(\frac{1}{2}(b_{i,j}-x_j)\times\left(a_{i+1,j}-a_{i,j}\right)+
\psi_{i,j}\left(1-\left|\frac{b_{i,j}-x_j}{2}\right|^2\right)\frac{b_{i,j}-x_j}{|b_{i,j}-x_j|}\right).
\end{align}

\end{thm}

\par  The main virtue of these formulae is that they show that the perimeter and volume of $B(X)$ are both computable in terms of $X$. One only needs to orient the vertices on each face of a Reuleaux polyhedron $B(X)$ and to choose corresponding adjacent vertices in order to quickly approximate the perimeter and volume of $B(X)$.  However, we note that the perimeter of $B(X)$ can be computed once its dual edge pairs and the distance between the endpoints of its edges are known.  This approach was developed by Bogosel \cite{bogosel2023volume} at the same time this article was being written and does indeed appear to be a superior method of computation.

\section{Perimeter formula}\label{PerSect} 
This section is dedicated to proving Theorem \ref{PerThm}. Let us fix $j\in \{1,\dots, m\}$ and denote the edge of $B(X)$ which joins $a_{i,j}$ to $a_{i+1,j}$ as $e_{i,j}$ for $i=1,\dots, N_j$ (where $a_{N_j+1,j}=a_{1,j}$).   According to the Gauss--Bonnet theorem,
\be\label{GBformula}
\sigma(F_j)=2\pi-\sum^{N_j}_{i=1}\left\{\int_{e_{i,j}}k_gds+\theta_{i,j}\right\}.
\ee
Here $F_j:=\partial B(x_j)\cap B(X)$ is the face opposite $x_j$, $k_g$ is the geodesic curvature of $e_{i,j}$ in $\partial B(x_j)$, and $\theta_{i,j}$ is the change in angle at the point in which $e_{i,j}$ connects to $e_{i+1,j}$. 
\par It is routine to verify $k_g=\sqrt{1-r^2}/r$ along $e_{i,j}$, where 
$$
r=\sqrt{1-\left|\frac{b_{i,j}-x_j}{2}\right|^2}
$$
is the radius of the circle including $e_{i,j}$; recall \eqref{EdgeCircle} and \eqref{beyejay}. As a result, the length $e_{i,j}$ is $r\psi_{i,j}$, where 
$$
\psi_{i,j}=\angle\left(a_{i,j}-\frac{b_{i,j}+x_j}{2},a_{i+1,j}-\frac{b_{i,j}+x_j}{2}\right).
$$
Consequently, 
\begin{align}
\int_{e_{i,j}}k_gds=\frac{\sqrt{1-r^2}}{r}\; r\psi_{i,j}=\sqrt{1-r^2}\psi_{i,j}=\left|\frac{b_{i,j}-x_j}{2}\right|\psi_{i,j}.
\end{align}
In view of \eqref{GBformula}, the surface area of the face of $\partial B(X)$ opposite $x_j$ is   
\be\label{sigmaFjay}
\sigma(F_j)=2\pi -\sum^{N_j}_{i=1}\left(\left|\frac{b_{i,j}-x_j}{2}\right|\psi_{i,j}+\theta_{i,j}\right).
\ee

\par We are left to verify each $\theta_{i,j}$ is given by the formula in the statement of 
the theorem. With this in mind, we fix $j\in \{1,\dots,m\}$ and $i\in \{1,\dots, N_j\}$ and parametrize the edge $e_{i,j}$ via
\begin{align}\label{gammapath}
\gamma(t)&=\frac{b_{i,j}+x_j}{2}+ \cos t\left(a_{i,j}-\frac{b_{i,j}+x_j}{2}\right)+\sin t\; \frac{b_{i,j}-x_j}{|b_{i,j}-x_j|}\times\left(a_{i,j}-\frac{b_{i,j}+x_j}{2} \right)
\end{align}
for $t\in [0,\psi_{i,j}]$. Direct computation yields 
\be\label{dotgammaformula}
\dot\gamma(t)=\frac{b_{i,j}-x_j}{|b_{i,j}-x_j|}\times\left(\gamma(t)-\frac{b_{i,j}+x_j}{2}\right).
\ee
We also can write an analogous parametrization $\eta$ of $e_{i+1,j}$ and compute  
\begin{align}
\theta_{i,j}&=\angle\left(\dot\gamma(\psi_{i,j}),\dot\eta(0)\right) \\ 
&=\angle\left((b_{i,j}-x_j)\times \left(a_{i+1,j}-\frac{b_{i,j}+x_j}{2}\right),(b_{i+1,j}-x_j)\times \left(a_{i+1,j}-\frac{b_{i+1,j}+x_j}{2}\right)\right),
\end{align}
which concludes our proof.

\section{Volume formula}\label{VolSect}
Again we will denote 
$F_j=\partial B(x_j)\cap B(X)$ as the face of $B(X)$ opposite $x_j$ with outward unit normal 
$$
N(x)=x-x_j\quad (x\in F_j).
$$
We note that $\partial F_j\subset \partial B(x_j)$ is the union of circular arcs, so it is piecewise smooth with outward normal $\nu$ tangent to $\partial B(x_j)$ that is defined at all but finitely many points on $\partial F_j$. 
Integrating $2N=-\Delta_{\partial B(x_j)}x$ over $F_j$ and integrating by parts leads to the formula
\be\label{IntegralFormulaNormal}
\int_{F_j} N d\sigma= -\frac{1}{2}\int_{\partial F_j} \nu ds.
\ee
See also page 571 of \cite{MR0812455}.

\par Now we are ready to prove Theorem \ref{VolThm}.
\begin{proof}[Proof of Theorem \ref{VolThm}] By \eqref{divThmFormula},
\begin{align}
V(B(X))&=\frac{1}{3}\sum^m_{j=1}\int_{F_j}x\cdot N d\sigma\\
&=\frac{1}{3}\sum^m_{j=1}\int_{F_j}(N+x_j)\cdot N d\sigma\\
&=\frac{1}{3}\sum^m_{j=1}\int_{F_j}(1+x_j\cdot N)d\sigma\\
&=\frac{1}{3}\sum^m_{j=1}\sigma(F_j)+\frac{1}{3}\sum^m_{j=1}x_j\cdot \int_{F_j} N d\sigma.
\end{align}
In view of \eqref{IntegralFormulaNormal}, 
\be\label{2ndToLastVolFormula}
V(B(X))=\frac{1}{3}P( B(X))+\frac{1}{6}\sum^m_{j=1}x_j\cdot \left( -\int_{\partial F_j} \nu ds\right).
\ee
\par Now fix $j\in\{1,\dots,m\}$ and $i\in \{1,\dots, N_j\}$, and let us once again parametrize the edge $e_{i,j}\subset \partial F_j$ by $\gamma(t)$ in \eqref{gammapath} for 
 $t\in [0,\psi_{i,j}]$. By  \eqref{dotgammaformula}, 
 \begin{align}
 \dot \gamma(t)\times(\gamma(t)-x_j)&=
 -\dot \gamma(t)\times\left(\frac{x_j-b_{i,j}}{2}\right)+ \dot \gamma(t)\times\left(\gamma(t)-\frac{x_j+b_{i,j}}{2}\right)\\
 &=-\left(\frac{b_{i,j}-x_j}{2}\right)\times \dot \gamma(t)-\left(1-\left|\frac{b_{i,j}-x_j}{2}\right|^2\right)\frac{b_{i,j}-x_j}{|b_{i,j}-x_j|}.
 \end{align}
It follows that 
\begin{align}
\int_{e_{i,j}}\nu ds&=\int^{\psi_{i,j}}_0\nu(\gamma(t))|\dot\gamma(t)|dt\\
&=\int^{\psi_{i,j}}_0\frac{\dot\gamma(t)}{|\dot\gamma(t)|}\times N(\gamma(t))| \dot\gamma(t)|dt\\
&=\int^{\psi_{i,j}}_0 \dot \gamma(t)\times(\gamma(t)-x_j)dt\\
&=-\left(\frac{b_{i,j}-x_j}{2}\right)\times  (a_{i+1,j}-a_{i,j})-\psi_{i,j}\left(1-\left|\frac{b_{i,j}-x_j}{2}\right|^2\right)\frac{b_{i,j}-x_j}{|b_{i,j}-x_j|}.
\end{align}
As $\partial F_j=\bigcup_{j=1}^{N_j}e_{i,j}$ and each edge overlaps just at the endpoints, 
$$
-\int_{F_j}\nu ds=\sum^{N_j}_{i=1}\left(\frac{1}{2}(b_{i,j}-x_j)\times\left(a_{i+1,j}-a_{i,j}\right)+
\psi_{i,j}\left(1-\left|\frac{b_{i,j}-x_j}{2}\right|^2\right)\frac{b_{i,j}-x_j}{|b_{i,j}-x_j|}\right).
$$
We now conclude by \eqref{2ndToLastVolFormula}. 
\end{proof}

\section{Numerical examples}
We will present a few subsets $X=\{x_1,\dots, x_m\}\subset \R^3$ below which are approximately extremal.  Here we mean 
that $|x_i-x_j|\le 1+\epsilon$ for all $i,j=1,\dots, m$ and $\{x_i,x_j\}$ is considered a diametric pair if $|x_i-x_j|\ge 1-\epsilon$, where $\epsilon$ is the machine epsilon used by 
\texttt{Mathematica} version 13.3.  We will express our coordinates in a finite decimal expansion ending with an apostrophe to emphasize this is a floating point representation; then we will plot the corresponding ball polyhedra $B(X)$ and indicate our perimeter and volume approximations of these shapes.

\begin{ex}\label{BullExample}
{\small \begin{align}
x_1&=(-0.454919332347376`, -0.33326312894833937`, -0.06582850639472673`), \\
x_2&=(-0.40156925513088615`, 0.24090106874115003`,  0.15688873018460048`),\\
x_3&=( 0.07053411741746762`, -0.366760736659379`, -0.4817547821803494`),\\
x_4&= (0.137577027658032`, 0.4594339691017592`, 0.07762662290811241`),\\
x_5&= (0.1725962059626301`, 0.17122989250966542`, -0.6588808147787987`),\\
x_6&= (0.21705359365165164`, -0.5368947429173325`, 0.045805669681666425`),\\
x_7&= (0.3186220388070483`, 0.16138379062317262`, 0.3303509631275429`),\\
x_8&= (0.495725494415533`, -0.1969613868194246`, 0.21291196149887548`).
\end{align}}
\begin{figure}[h]
\centering
   \includegraphics[width=.55\textwidth]{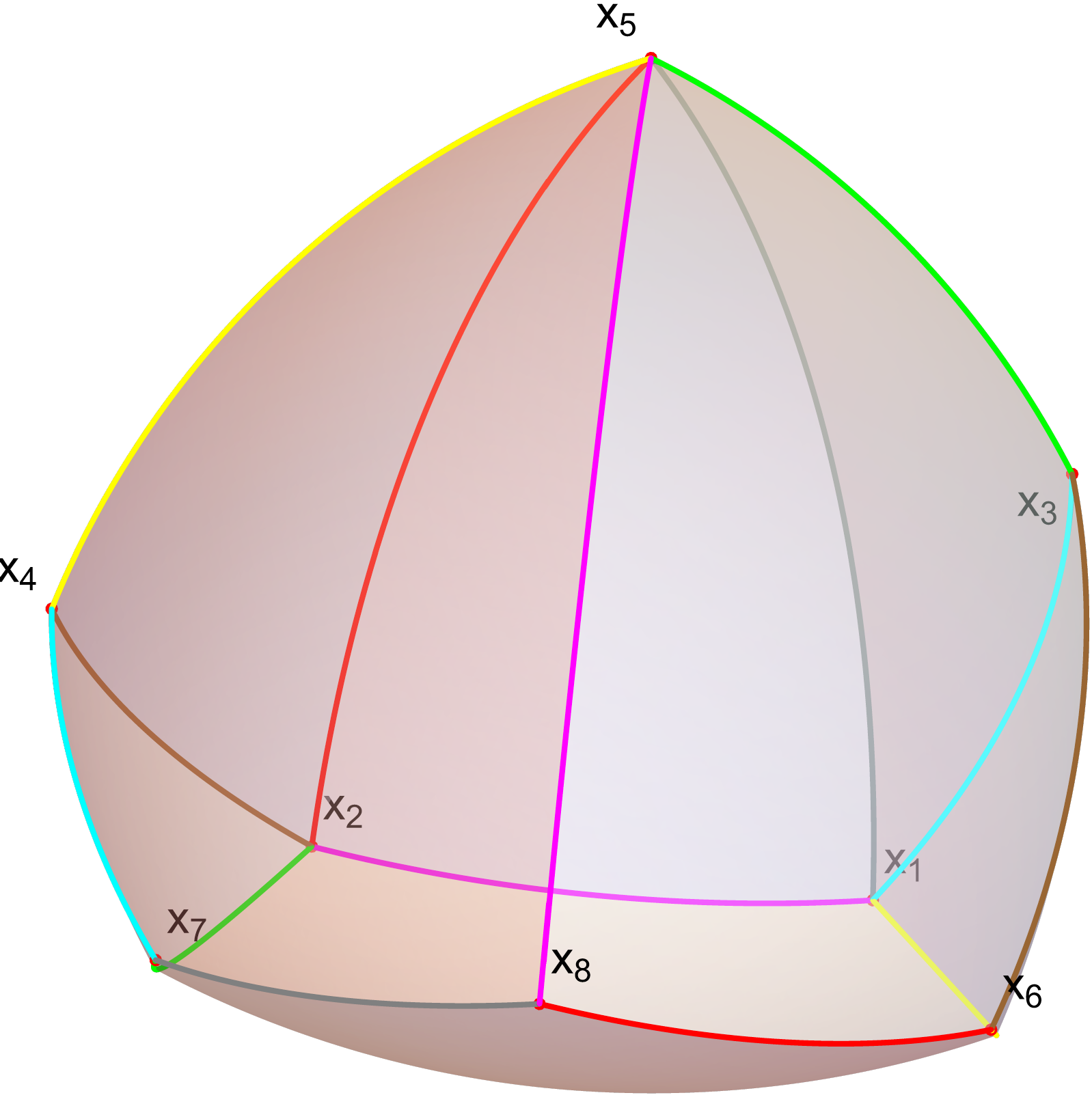} 
 \caption*{ $P(B(X))\approx 3.004217845729678$, $V(B(X))\approx 0.45147884098820945$ }
\end{figure}
\end{ex}
\newpage

\begin{ex}
{\small \begin{align}
x_1&=(-0.2886751345948129`, 0.5`, 0.`), \\
x_2&=(-0.2886751345948129`, -0.5`, 0.`),\\
x_3&=(0.5773502691896258`, 0.`, 0.`),\\
x_4&= (0.`, 0.`, 0.816496580927726`),\\
x_5&= (-0.2210038810439029`, 0.44313191935026985`, 0.325457807012713`),\\
x_6&= (-0.18703718866131322`, 0.2888224121457739`, 0.5764489817169214`),\\
x_7&= (-0.02128367887294722`, -0.45074924952667716`, -0.0759002585247513`),\\
x_8&= (0.12268980214227843`, -0.4088095374975887`, -0.06959966506175008`).
\end{align}}

\begin{figure}[h]
\centering
      \includegraphics[width=.6\textwidth]{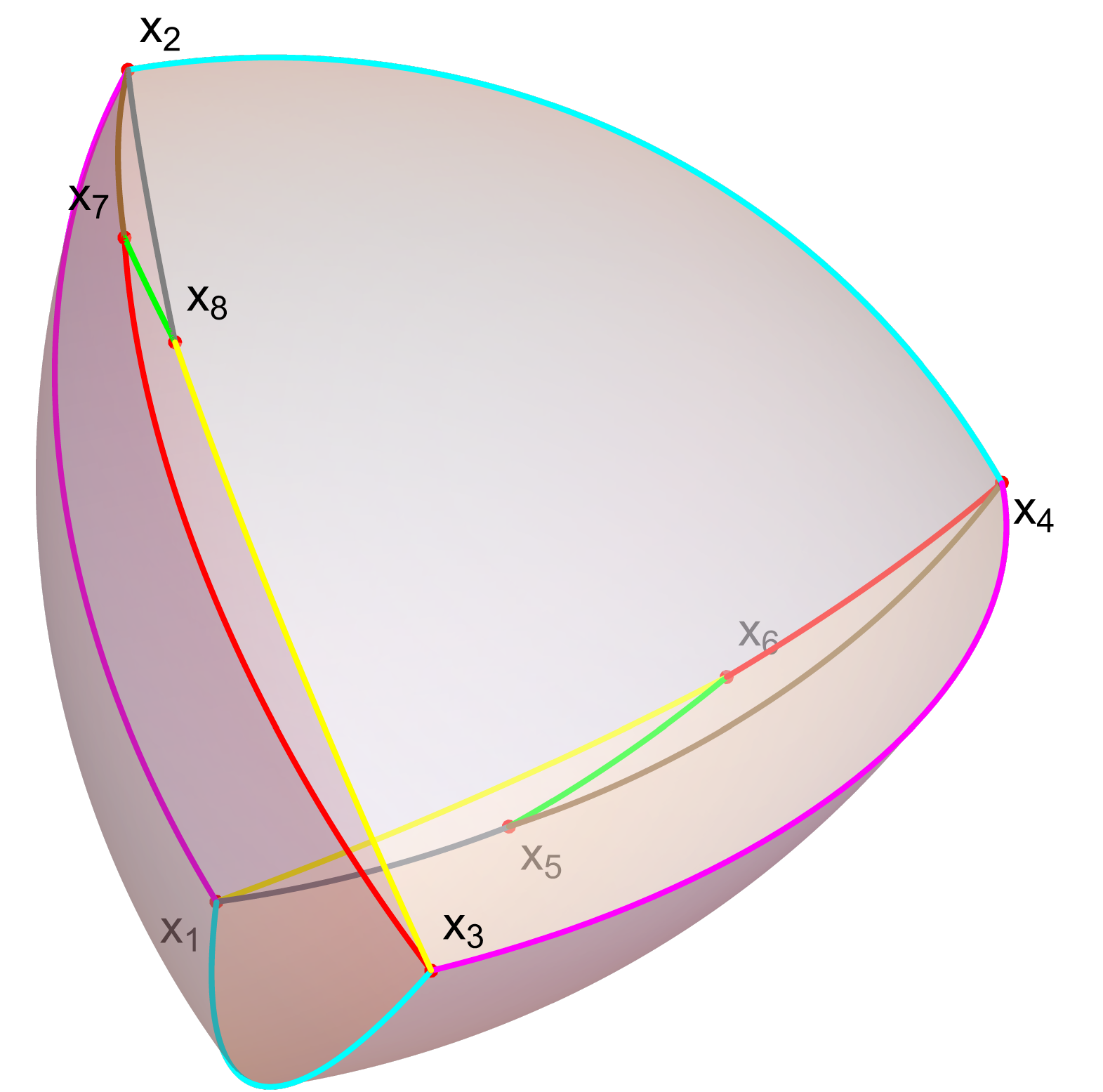} 
 \caption*{ $P(B(X))\approx 2.96308631315525$, $V(B(X))\approx 0.42168401162294744$ }
\end{figure}
\end{ex}

\newpage

\begin{ex}
{\small \begin{align}
x_1&=(0.19889705322366025`, -0.509209185741388`, 0.43182837977567157`),\\
x_2&= (-0.22000635309878455`, -0.3484602553212368`, 0.49776149424867744`),\\
x_3&= (-0.04501256638857934`, -0.4905973154851182`, -0.07629144701261002`),\\
x_4&=( 0.7474800519578605`, -0.27277466456492816`,  0.2564282238891228`),\\
x_5&=( 0.0446997445042288`, -0.0578646460428892`, -0.42174102002734665`),\\
x_6&=(0.2680807002029154`, 0.4385358535711189`, 0.12039144643412861`),\\
x_7&= (0.06662773385191895`, 0.4586282815917317`, 0.2178296321645024`), \\
x_8&=(0.681203847537043`, 0.055629677150635895`, 0.34113614270937886`),\\
x_9&= (0.04735982412164459`, 0.06547174433710401`, -0.3723968976187925`), \\
x_{10}&=(-0.015302651388321266`, 0.3264540060426252`, 0.49950752740506427`).
\end{align}}
\begin{figure}[h]
\centering
      \includegraphics[width=.6\textwidth]{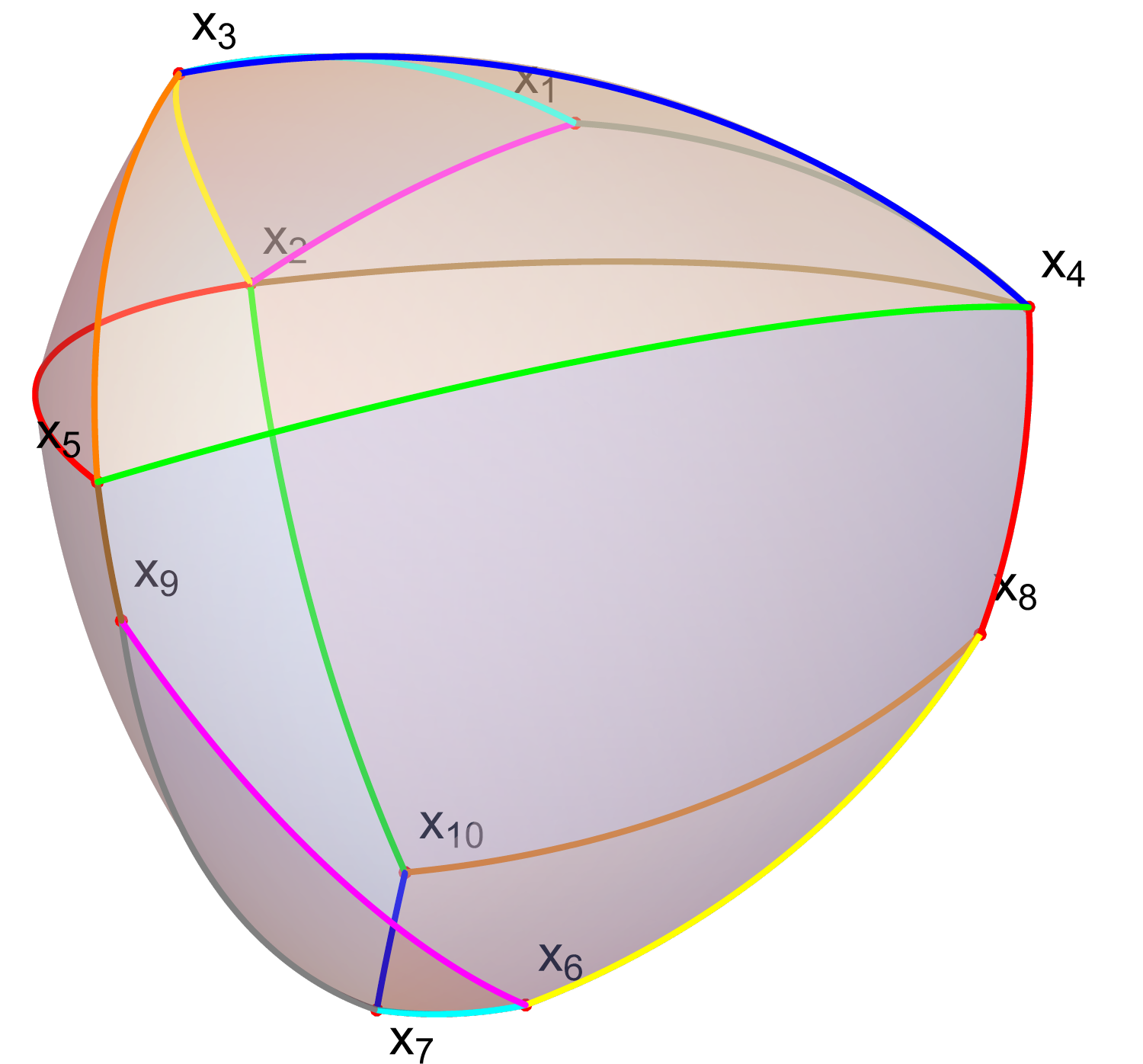} 
 \caption*{ $P(B(X))\approx 3.0006801203477895$, $V(B(X))\approx 0.4499825760002685$ }
\end{figure}

\end{ex}

\newpage 

\begin{ex}
{\small \begin{align}
x_1=&(-0.4492911671196649`, 0.051756018689026004`, 0.40124302234607867`),\\
x_2= & (-0.3444292318383386`, 0.31092146504228`, 0.381659938710443`),\\
x_3= & (-0.17073076130547793`, 0.4731052458004084`, 0.4349796025348303`), \\
x_4=  &(-0.16667966146049476`,  0.26258781488620714`, -0.02222637246905917`), \\
x_5=& (0.017008554645613436`, -0.5024995134353547`,   0.3211845982387496`), \\
x_6=&  (0.13868993504984`, -0.42799467084581205`,  0.13120703483464335`),\\
x_7=  & (0.1473454770808581`, 0.3483283196322836`, 0.761482923670077`),\\
x_8= &  (0.2929193940879094`, 0.2877251625593002`, -0.2260065659145016`), \\
x_9= & (0.3918033569102077`, -0.34691391562342916`, 0.5404496965949347`), \\ 
x_{10}= & (0.41559079487731293`, -0.1802849044395009`, -0.043881905967219983`),\\
 x_{11}=&(0.4550973937486081`, 0.2816385958654753`,  0.7607361817821491`),\\
x_{12}= &  (0.5103125846259885`, -0.1720971456961624`, 0.571684802524544`).
\end{align}}
\begin{figure}[h]
\centering
      \includegraphics[width=.6\textwidth]{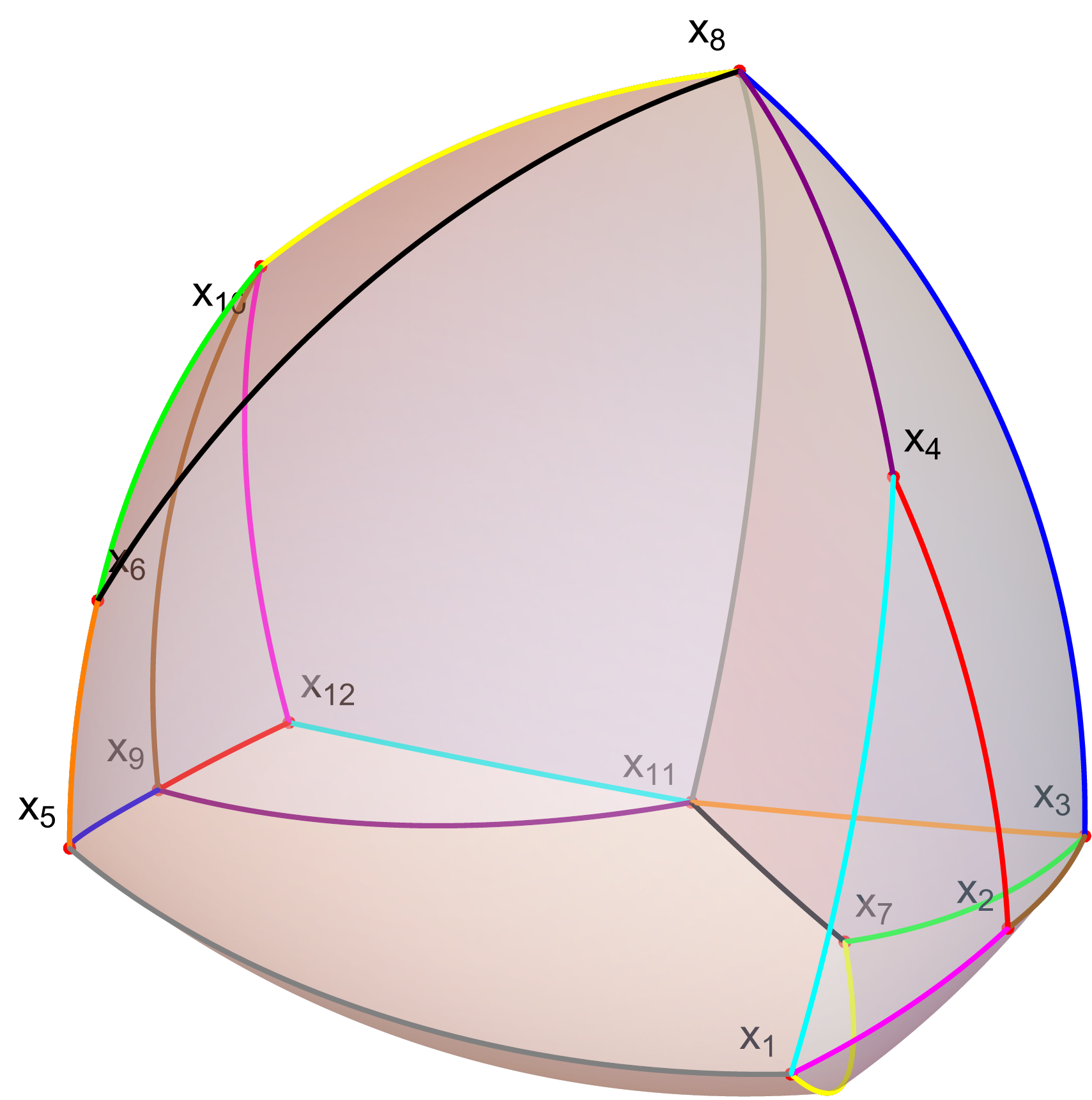} 
 \caption*{ $P(B(X))\approx 3.0036684374206386$, $V(B(X))\approx0.45172374549885497$}
\end{figure}
\end{ex}
\newpage

\section{Pyramids}
A Reuleaux tetrahedron is member of a family pyramids. In general, a member of this group has $n+1$ vertices where $n\ge 3$ and $n$ is odd. 
The first $n$ vertices $\{x_1,\dots, x_n\}$ are the vertices of a regular polygon with diameter one which form the base and the $(n+1)$st vertex $x_{n+1}$ is the apex. Such a collection of points 
$$
Z_n=\{x_1,\dots,x_{n+1}\}
$$
may be expressed explicitly as  
\be\label{PyramidBasePoints}
x_j=\left( \frac{\cos\left(\frac{2\pi j}{n}\right)}{2\cos\left(\frac{\pi}{2n}\right)}, \frac{\sin\left(\frac{2\pi j}{n}\right)}{2\cos\left(\frac{\pi}{2n}\right)} , 0 \right)
\ee
for $j=1,\dots, n$ and 
$$
x_{n+1}=\left(0,0, \sqrt{1-\left(2\cos\left(\frac{\pi}{2n}\right)\right)^{-2}}\right).
$$
A Reuleaux tetrahedron can be designed with these coordinates for $n=3$. 

\par First we will need to establish that $Z_n$ is extremal. We also note that this construction was presented in Example 1.2 of \cite{MR2593321}.
\begin{lem}
For each odd $n\ge 3$, $Z_n$ is extremal.  
\end{lem}
\begin{proof}
Direct computation gives 
\begin{align}
|x_j-x_1|=\frac{\sin\left(\frac{\pi}{n}(j-1)\right)}{\cos\left(\frac{\pi}{2n}\right)}
\end{align}
for $j=1,\dots, n$. Since $\sin$ is increasing on $[0,\pi/2]$, 
$$
\sin\left(\frac{\pi}{n}(j-1)\right)\le \sin\left(\frac{\pi}{n}\frac{n-1}{2}\right)= \sin\left(\frac{\pi}{2}-\frac{\pi}{2n}\right)=\cos\left(\frac{\pi}{2n}\right).
$$
for $j=1,\dots,  1+(n-1)/2$. As $\sin$ is decreasing on $[\pi/2,\pi]$, we can argue similarly to conclude the above inequality for $j=1+(n+1)/2,\dots, n$.  Thus $|x_j-x_1|\le 1$ for $j=1,\dots, n$ and equality holds for $j=1+(n-1)/2$.   

\par Likewise, we find 
$$
|x_j-x_k|=|x_{j-k+1}-x_1|\le 1
$$
for $1\le k\le j\le n$. It is also routine to check that there are $n$ diametric pairs  
$$
\{x_1, x_{1+(n-1)/2}\}, \{x_2, x_{2+(n-1)/2}\},\dots, \{x_{n-1},x_{(n-3)/2}\},\{x_n,x_{(n-1)/2}\}
$$
among $\{x_1,\dots, x_n\}$.  Moreover, 
$$
|x_{n+1}-x_j|=1\quad j=1,\dots, n.
$$
Therefore, $\{x_1,\dots, x_{n+1}\}$ has $2n=2((n+1)-1)$ diametric pairs and is thus extremal. 
\end{proof}
\begin{figure}[h]
\centering
   \includegraphics[width=.44\textwidth]{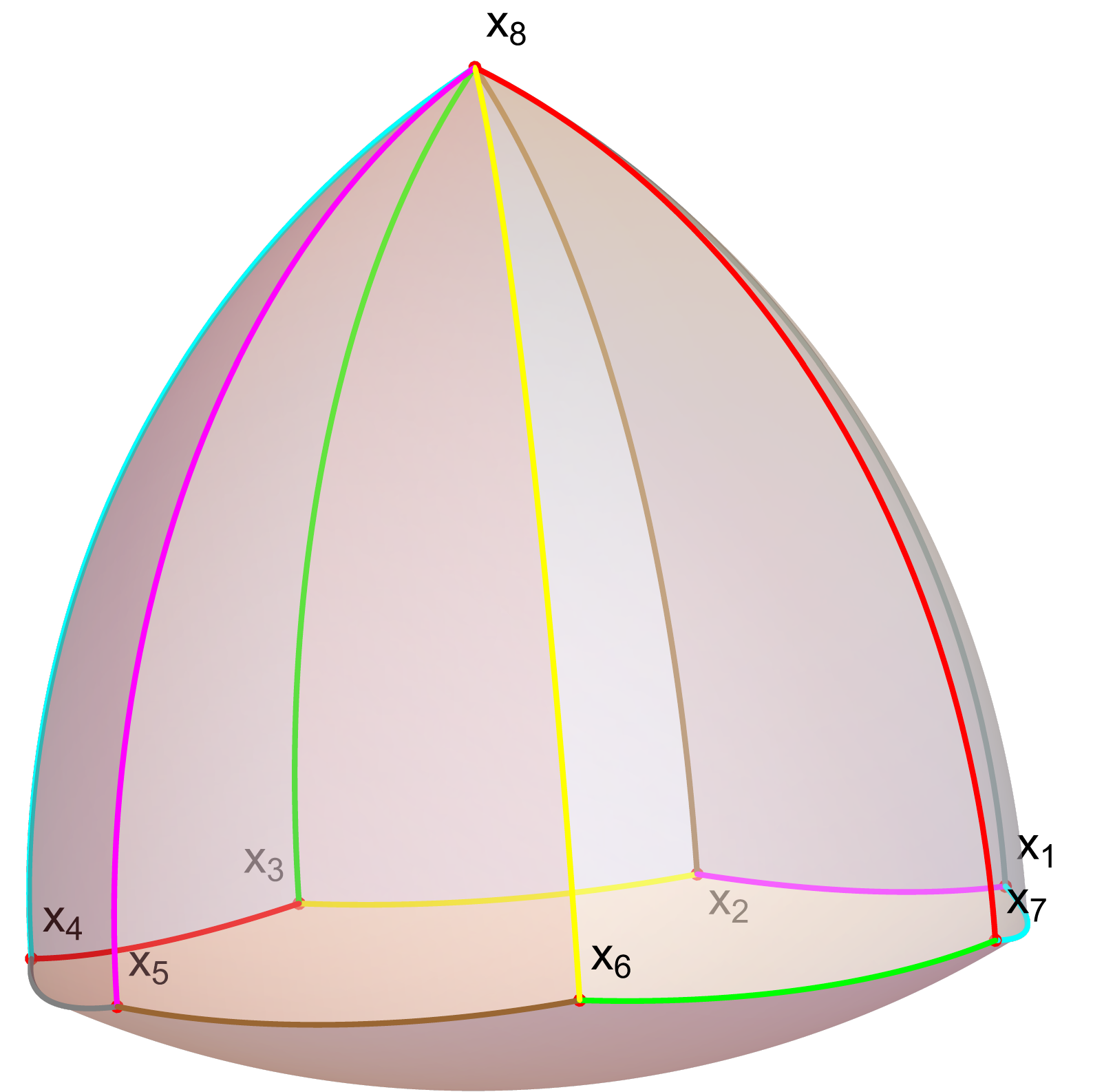} 
    \hspace{.3in}
      \includegraphics[width=.44\textwidth]{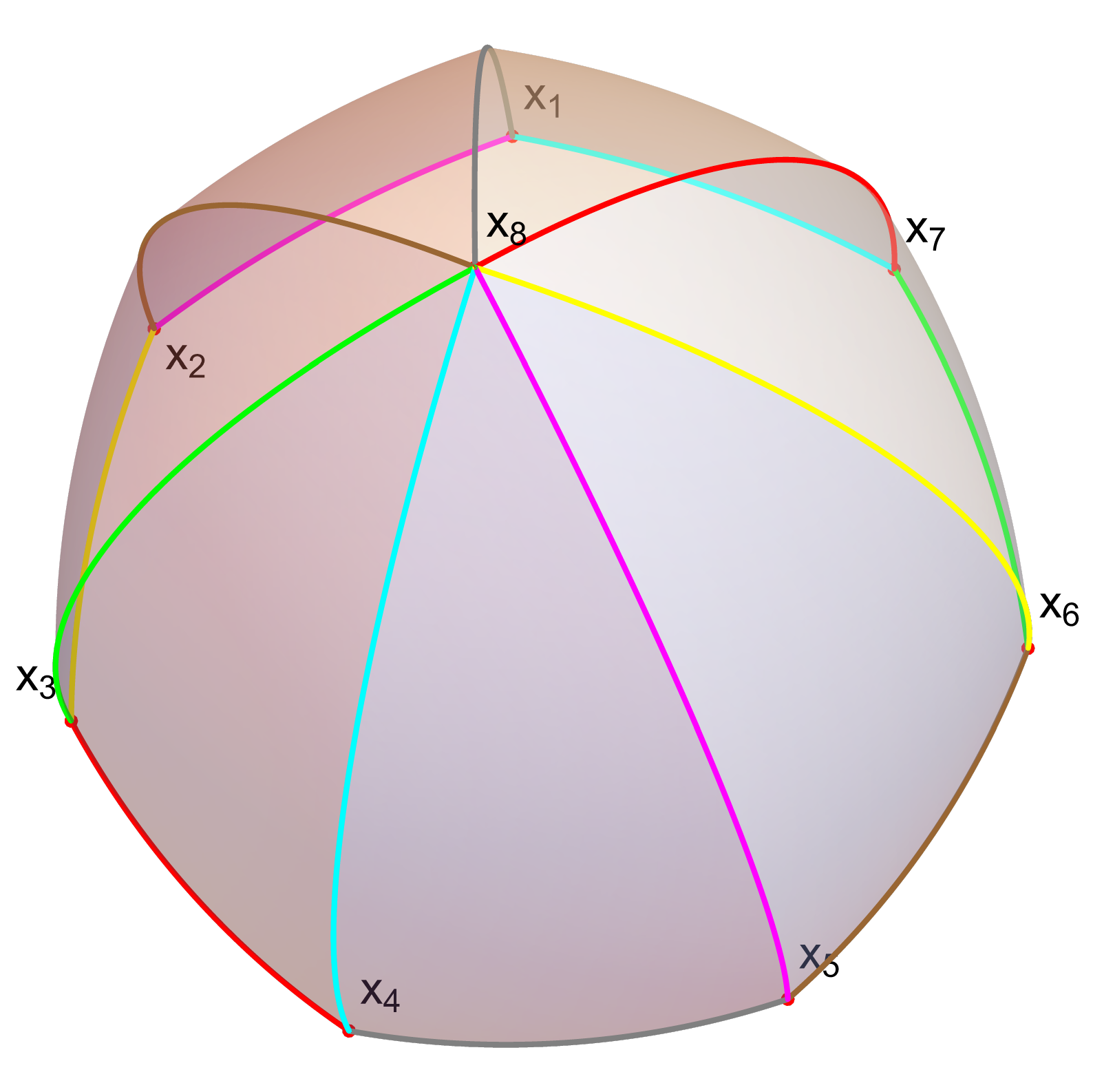} 
 \caption{A Reuleaux pyramid with $7$ base vertices.}\label{ReuleauxSept}
\end{figure}
\par We denote $R_n=B(Z_n)$ as a {\it Reuleaux pyramid} with $n$ base vertices.
See Figures \ref{ReuleauxTetra}, \ref{ReulauxTetraEdge}, and \ref{ReuleauxSept} for Reuleaux pyramids with $n=3,5,$ and $7$ base vertices, respectively.  It is a straightforward task to compute the perimeter and volume of a given Reuleaux pyramid by employing Theorems \ref{PerThm} and \ref{VolThm}. In general, we used \texttt{Mathematica} to find  
\begin{align}
P(R_n) &=2\pi(n+1)-n\left[2\cos ^{-1}\left(\frac{1}{3} \left(4
   \cos \left(\frac{\pi
   }{n}\right)-1\right)\right)+2 \sin
   \left(\frac{\pi }{2 n}\right) \cos^{-1}\left(\frac{\cos \left(\frac{\pi
   }{n}\right)}{\cos \left(\frac{\pi
   }{n}\right)+1}\right)\right.\\
   &\quad \quad  + 2 \cos
   ^{-1}\left(\frac{1}{\sqrt{3}}\tan \left(\frac{\pi }{2
   n}\right)\right)+\cos ^{-1}\left(5-4
   \cos \left(\frac{\pi }{n}\right)-2 \sec\left(\frac{\pi }{2 n}\right)^2\right)\Bigg]
\end{align}
and 
\begin{align}
V(R_n)&=\frac{2\pi}{3}(n+1)-\frac{n}{3}\Bigg[\frac{19}{8}\cos ^{-1}\left(\frac{1}{3} \left(4
   \cos \left(\frac{\pi
   }{n}\right)-1\right)\right) +2 \cos
   ^{-1}\left(\frac{1}{\sqrt{3}}\tan \left(\frac{\pi }{2
   n}\right)\right)\\
   &\quad +\left(\frac{1}{2}\sin
   \left(\frac{\pi }{ n}\right)\cos
   \left(\frac{\pi }{2 n}\right)+2 \sin
   \left(\frac{\pi }{2 n}\right)\right) \cos^{-1}\left(\frac{\cos \left(\frac{\pi
   }{n}\right)}{\cos \left(\frac{\pi
   }{n}\right)+1}\right)\\
   &\left.\quad+\cos ^{-1}\left(5-4
   \cos \left(\frac{\pi }{n}\right)-2 \sec\left(\frac{\pi }{2 n}\right)^2\right)-\frac{1}{4} \sin \left(\frac{\pi
   }{n}\right)\sqrt{4-\sec\left(\frac{\pi }{2 n}\right)^2} \right]
\end{align}
for each odd $n\ge 3.$  Substituting $n=3$, we arrive at Harbourne's formulae  for the perimeter and volume of a Reuleaux tetrahedron mentioned in the introduction. 

\par We close this section with some numerical results obtained from the formulae above. 
\begin{center}
\begin{tabular}{||c |l|} 
 \hline
$n$ & $P(R_n)$ \\
 \hline\hline
 3 & 2.9754717165844013  \\
 \hline
5 & 2.987479950727929 \\
 \hline
7 & 2.9904113459590356 \\
\hline
9 & 2.9915766744579857 \\
\hline 
11 & 2.9921578300569203 \\ 
\hline
13 & 2.992489472713454 \\ 
\hline
15 & 2.9926965856128667 \\
\hline 
17 & 2.992834591165926 \\ 
\hline
19 & 2.9929311619533574 \\ 
 \hline
21 & 2.9930013781619644 \\ 
 \hline
 23 & 2.9930540308850655\\
  \hline
 25 & 2.9930945274828598\\
   \hline
\end{tabular}
\hspace{.5in}
\begin{tabular}{||c |l|} 
 \hline
$n$ & $V(R_n)$ \\
 \hline\hline
 3 & 0.4221577331158264  \\
 \hline
5 & 0.44065107464468123 \\
 \hline
 7& 0.44508906045965013\\
  \hline
 9& 0.4468455644780944\\ 
 \hline
11 & 0.4477199306410655 \\
\hline
13 & 0.4482184208305563 \\
\hline 
15 & 0.448529556413958 \\ 
\hline
17 & 0.4487368010969238 \\ 
\hline
19 & 0.44888178739480866 \\
\hline 
21 & 0.44898718810471205 \\ 
\hline
23 & 0.4490662144272868 \\ 
 \hline
25 & 0.4491269898003859 \\ 
   \hline
\end{tabular}
\end{center}

\section{Elongated pyramids}
We will now discuss a variant of the Reuleaux pyramids.  In particular, we will define
a set of points which are the vertices of a polyhedron in $\R^3$ called an elongated pyramid.  This set of points will consist of the vertices of a regular polygon $\{x_1,\dots, x_n\}$ with diameter one, another set of vertices $\{x_{n+1},\dots, x_{2n}\}$ of a regular polygon in a parallel plane, and an apex vertex $x_{2n+1}$. 
\begin{prop}\label{EPpointsProp}
Fix $t\in (0,1)$ and $n\ge 3$ odd. Assume $x_j$ is given by \eqref{PyramidBasePoints} and set
$$
x_{n+j}=tx_j-\alpha e_n
$$
for $j=1,\dots, n$, 
where 
$$
\alpha=\sqrt{1-\frac{t^2+1+2t\cos\left(\frac{\pi}{n}\right)}{(2\cos(\pi/(2n)))^2}}.
$$
Define $x_{2n+1}=\beta e_n$, with 
$$
\beta=\sqrt{1-\frac{t^2}{(2\cos(\pi/(2n)))^2}}-\alpha.
$$
Then $\{x_1,\dots, x_{2n+1}\}\subset \R^3$ is extremal. 
\end{prop}
\begin{figure}[h]
\centering
   \includegraphics[width=.44\textwidth]{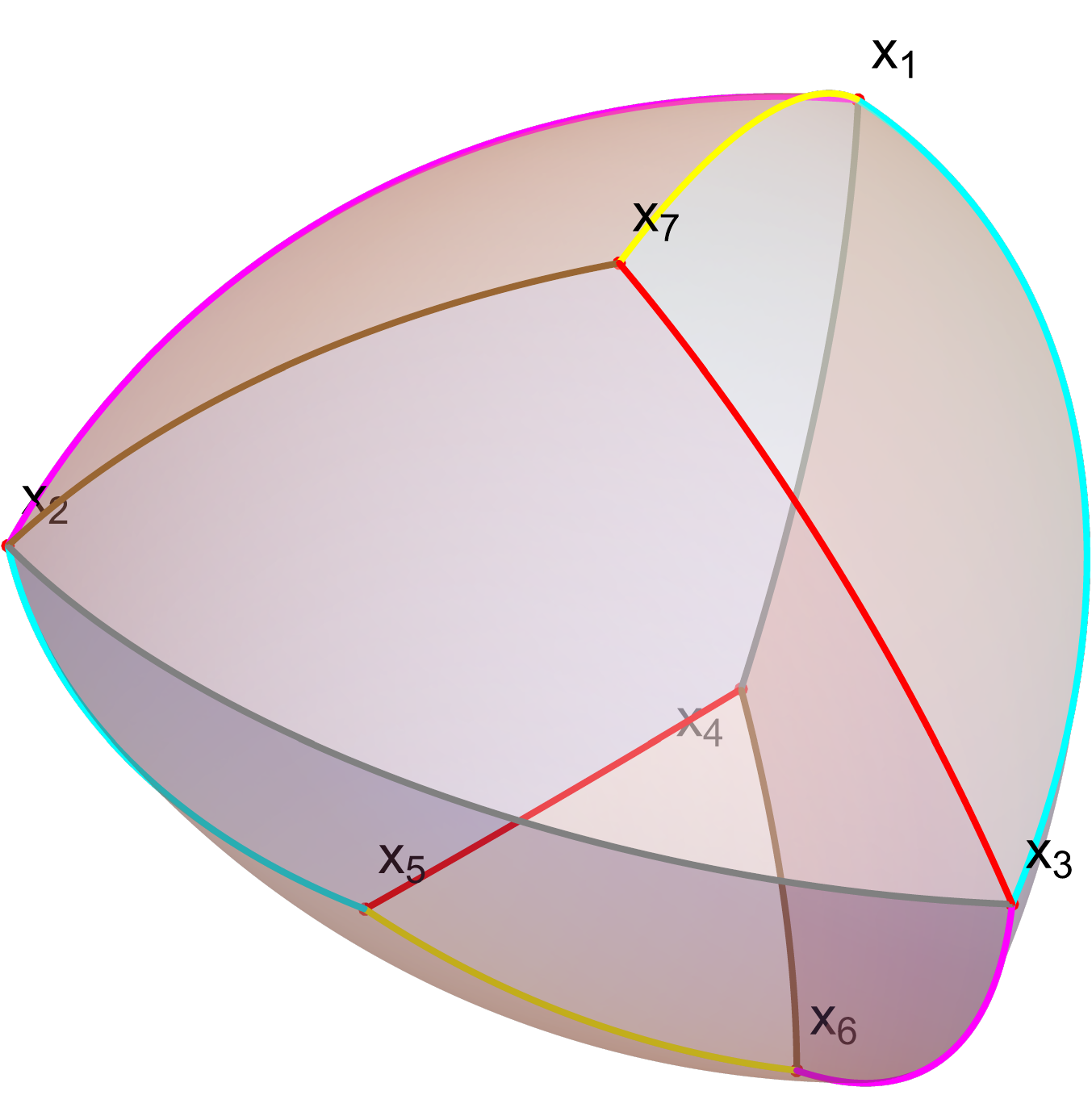} 
    \hspace{.3in}
      \includegraphics[width=.44\textwidth]{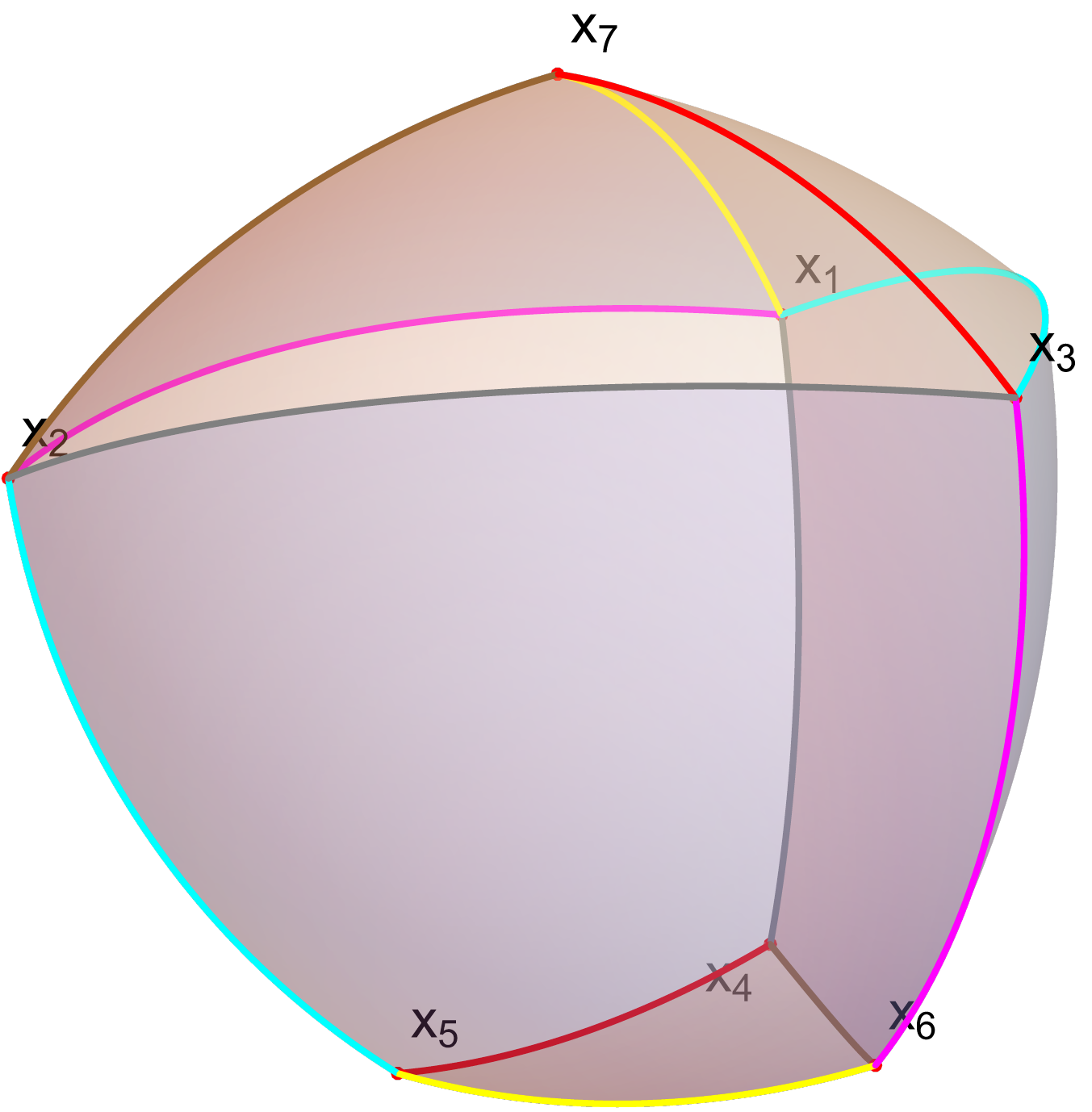} 
 \caption{Two views of the Reuleaux elongated pyramid $B(\{x_1,\dots, x_{2n+1}\})$, where $n=3$ and $t=1/2$.}\label{ReuleauxEP}
\end{figure}
\begin{proof}
It suffices to show that $\{x_1,\dots, x_{2n+1}\}$ has diameter one and that there are $4n=2((2n+1)-1)$ diametric pairs.   First, we recall that there are $n$ diametric pairs among the subset $\{x_1,\dots, x_n\}$.  Next, we fix $j=1,\dots, n$ and observe that for $i=1,\dots, n$,
\begin{align}
|x_j-x_{i+n}|^2&=|x_j-tx_i-\alpha e_n|^2\\
&=|x_j-tx_i|^2+\alpha^2\\
&=|x_j|^2+t^2|x_i|^2-2tx_i\cdot x_j+\alpha^2\\
&=\frac{t^2+1-2t\cos\left(\frac{2\pi}{n}(i-j)\right)}{(2\cos(\pi/(2n)))^2}+1-\frac{t^2+1+2t\cos\left(\frac{\pi}{n}\right)}{(2\cos(\pi/(2n)))^2}\\
&=1-2t\frac{\cos\left(\frac{2\pi}{n}(i-j)\right)+\cos\left(\frac{\pi}{n}\right)}{(2\cos(\pi/(2n)))^2}.
\end{align}
Since $\cos$ is decreasing on $[0,\pi]$,  
$$
\cos\left(\frac{2\pi}{n}k\right)\ge\cos\left(\frac{2\pi}{n}\frac{n-1}{2}\right)=-\cos\left(\frac{\pi}{n}\right)\;\;\text{for}\;\; k=0,\dots, \frac{n-1}{2}.
$$
And as $\cos$ is increasing on $[\pi,2\pi]$,
$$
\cos\left(\frac{2\pi}{n}k\right)\ge\cos\left(\frac{2\pi}{n}\frac{n+1}{2}\right)=-\cos\left(\frac{\pi}{n}\right)\;\;\text{for}\;\; k=\frac{n+1}{2},\dots,n.
$$
Thus, $|x_j-x_{i+n}|\le 1$ for all $i,j=1,\dots, n$ and there are two values of $i$ such that $|x_j-x_{i+n}|= 1$. It follows 
that $\{x_1,\dots, x_{2n}\}$ has diameter one and has $3n$ diametric pairs. 

\par Now we will add $x_{2n+1}= \beta e_n$ to our collection of points. Notice 
\begin{align}
|x_{2n+1}-x_{j+n}|^2&=(\beta+\alpha)^2+t^2|x_j|^2\\
&=1-\frac{t^2}{(2\cos(\pi/(2n)))^2}+t^2\frac{1}{(2\cos(\pi/(2n)))^2}\\
&=1
\end{align}
for $j=1,\dots, n$. And in view of the elementary inequality $\sqrt{a}-\sqrt{b}\le \sqrt{a-b}$ for $a\ge b$, 
$$
\beta^2\le \frac{1+2t\cos\left(\frac{\pi}{n}\right)}{(2\cos(\pi/(2n)))^2}.
$$
Since $t\in (0,1)$, it then follows that 
$$
|x_{2n+1}-x_{j}|^2=\beta^2+|x_j|^2< \frac{2(1+\cos\left(\frac{\pi}{n}\right))}{(2\cos(\pi/(2n)))^2}=1
$$
for $j=1,\dots, n$.  Therefore, $\{x_1,\dots, x_{2n+1}\}$ has diameter one and has $4n$ diametric pairs. We conclude that this set of points 
is extremal. 
\end{proof}
\begin{figure}[h]
\centering
   \includegraphics[width=.44\textwidth]{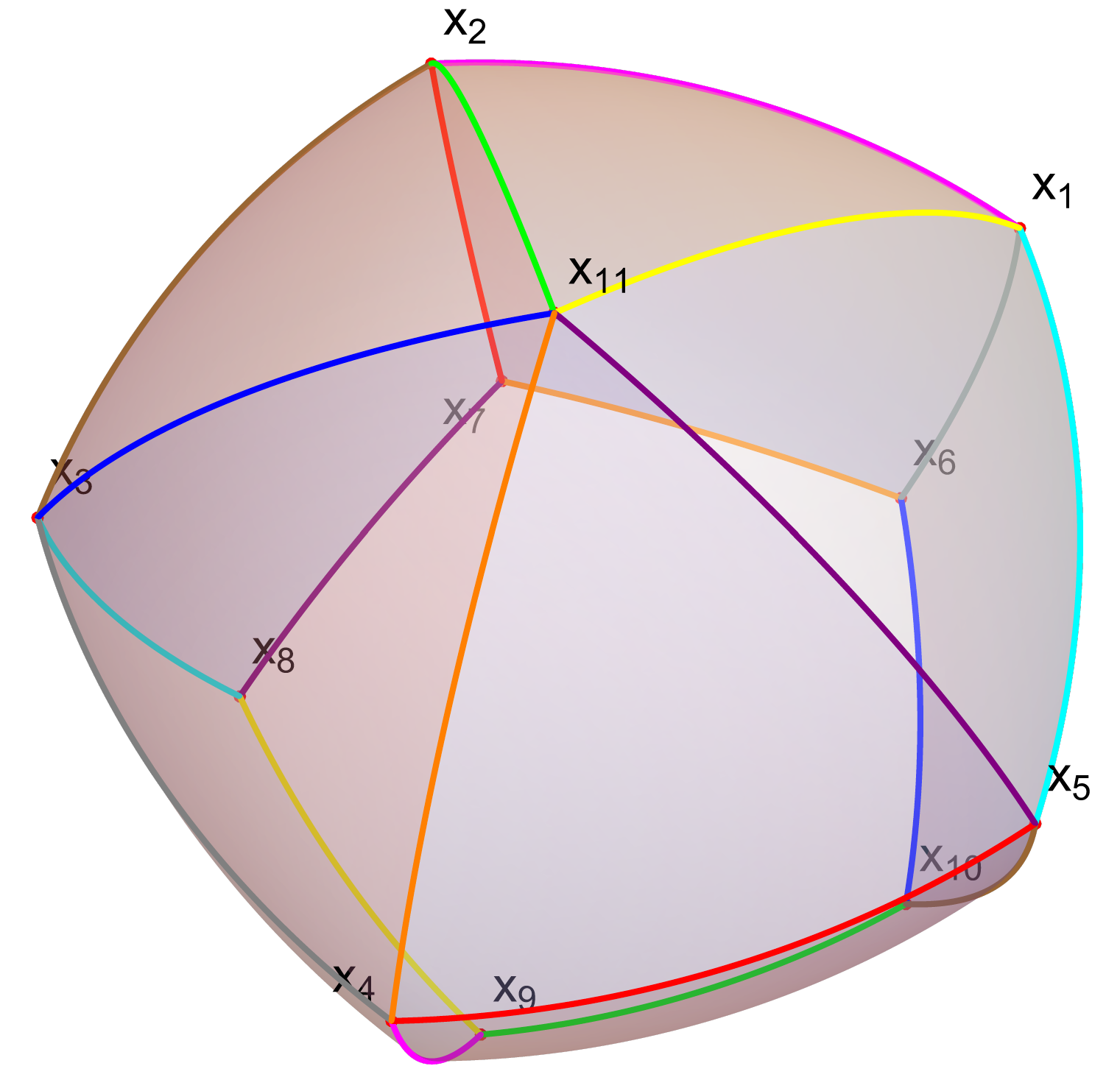} 
    \hspace{.3in}
      \includegraphics[width=.44\textwidth]{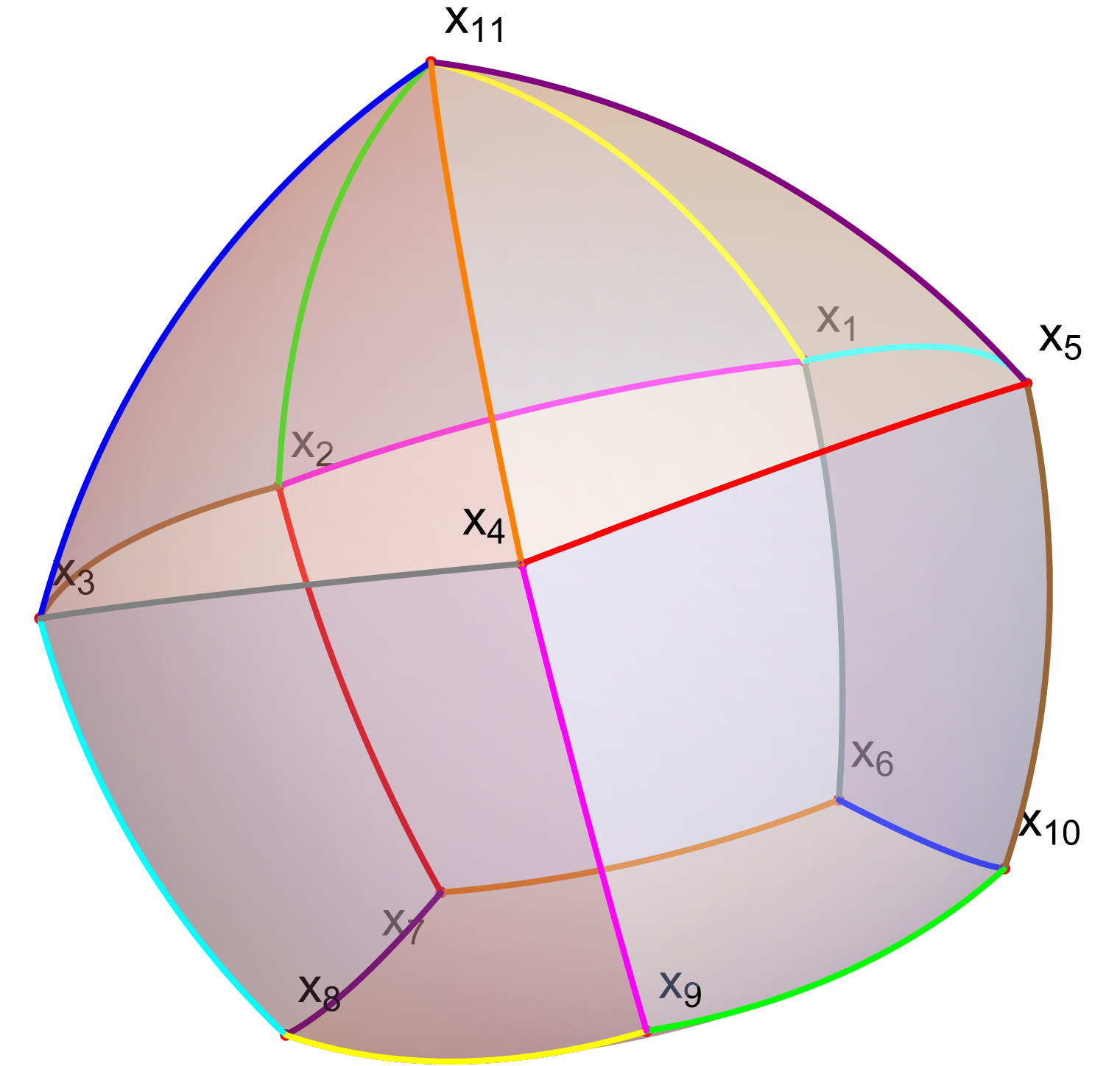} 
 \caption{The Reuleaux elongated pyramid $E_{5,3/4}$.}
\end{figure}
\par We will write $E_{n,t}=B(\{x_1,\dots, x_{2n+1}\})$
and call this shape a {\it Reuleaux elongated pyramid}. Given the symmetry of $E_{n,t}$, it is possible to compute its perimeter and volume explicitly. However, these formulae are cumbersome, so we will just present a few numerical computations in a table below. 
\begin{center}
\begin{tabular}{||c|c|l|} 
 \hline
$n$ & $t$ & $P(E_{n,t})$ \\
 \hline\hline
 3 & 1/2 & 2.9931270190442447  \\
 \hline
5 & 3/4 &3.0631372552821503 \\
 \hline
7 & 3/4 &3.0771714565245922 \\
\hline
9 & 4/5 & 3.078504604790311 \\
\hline 
11 &5/6& 3.0759025640806144 \\ 
\hline
13 &2/5 &3.0595001998462337 \\ 
\hline
15 & 1/5 &3.0268739279677703 \\
\hline 
17 &1/10& 3.0097623900581425 \\ 
   \hline
\end{tabular}
\hspace{.5in}
\begin{tabular}{||c |c |l|} 
 \hline
$n$ & $t$& $V(E_{n,t})$ \\
 \hline\hline
 3 & 1/2 &0.4434445124846693  \\
 \hline
5 & 3/4 & 0.48312463940394 \\
 \hline
 7& 3/4& 0.4888050312329182\\
  \hline
 9& 4/5& 0.4916683622752752\\ 
 \hline
11 &5/6& 0.49047575366738516 \\
\hline
13 &2/5 &  0.4821997310410573 \\
\hline 
15 &1/5&  0.4658188594948641 \\ 
\hline
17 & 1/10& 0.4572829425509739 \\ 
   \hline
\end{tabular}
\end{center}

\section{Diminished trapezohedra}\label{DiminishedTrapSect}
In this section, we will construct extremal sets whose vertices are the same as the vertices of polyhedra in $\R^3$ known as diminished trapezohedra.  
To this end, we suppose $n\ge 4$ is even and $x_1,\dots, x_n$ are the vertices of a regular polygon of diameter one centered at the origin.  For example, we may take
\be\label{OtherPyramidBasePoints}
x_j=\left( \frac{1}{2}\cos\left(\frac{2\pi j}{n}\right), \frac{1}{2}\sin\left(\frac{2\pi j}{n}\right) , 0 \right)
\ee
for $j=1,\dots, n$. Next we choose
\be
x_{n+j}=Tx_j+\sin\left(\frac{\pi}{2n}\right) e_3
\ee
for $j=1,\dots, n$ and 
\be\label{Tsymm}
T=
\left(
\begin{array}{ccc}
\cos\left(\frac{\pi}{n}\right)&-\sin\left(\frac{\pi}{n}\right)&0\\
\sin\left(\frac{\pi}{n}\right)&\cos\left(\frac{\pi}{n}\right)&0\\
0&0&1
\end{array}
\right).
\ee
Finally, we select $x_{2n+1}=(\sqrt{3}/2)e_3$. We will verify that $\{x_1,\dots, x_{2n+1}\}$  is extremal and say $D_n=B(\{x_1,\dots, x_{2n+1}\})$ is a {\it Reuleaux diminished trapezohedron}. 
\begin{figure}[h]
\centering
   \includegraphics[width=.44\textwidth]{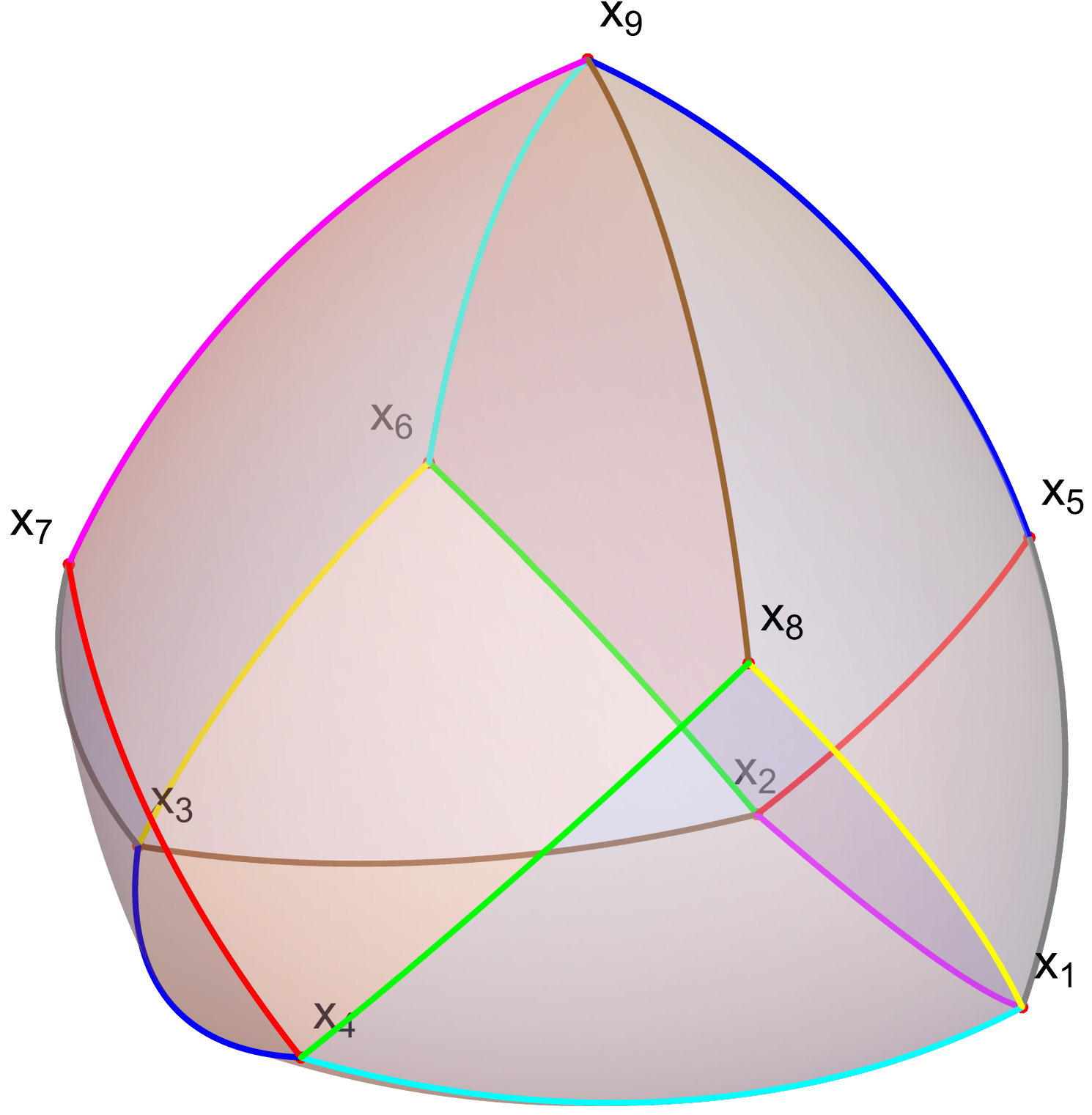} 
    \hspace{.3in}
      \includegraphics[width=.44\textwidth]{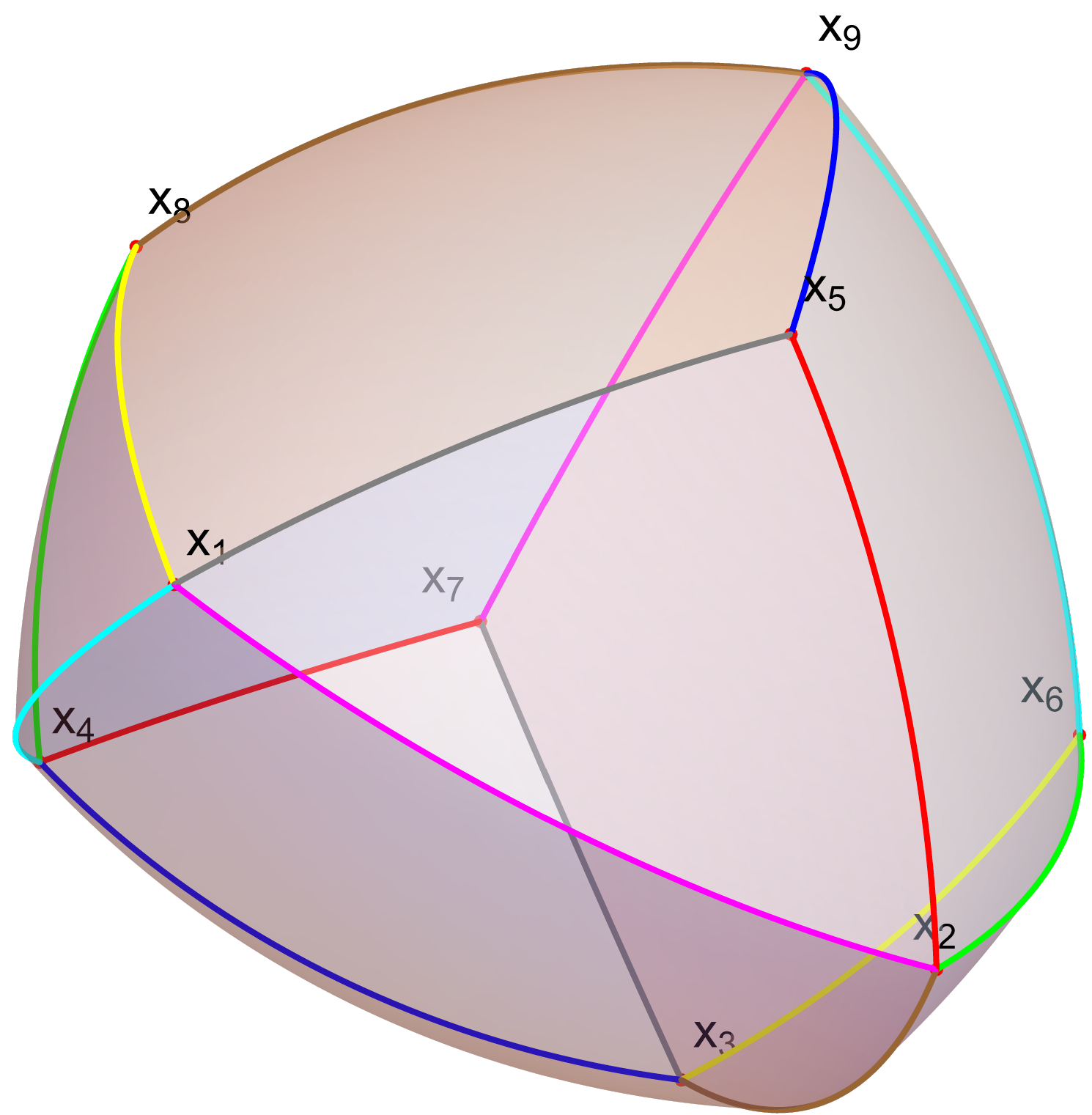} 
 \caption{Two views of the Reuleaux diminished trapezohedron $D_4$.}\label{DimTrapFig}
 \end{figure}
\begin{prop}
The set $\{x_1,\dots,x_{2n+1}\}\subset\R^3$ is extremal. 
\end{prop}
\begin{proof}
It suffices to show $\{x_1,\dots,x_{2n+1}\}\subset\R^3$ has diameter one and that this set has $4n$ diametric pairs.  First we recall that $\{x_1,\dots, x_n\}$ has diameter one and it has $n/2$ diametric pairs. 
This is a basic fact regarding the vertices of an even sided, regular polygon of diameter one.  It is also evident that 
$|x_{2n+1}-x_j|=1$ for $j=1,\dots, n$. Therefore, $\{x_1,\dots,x_n,x_{2n+1}\}$ has diameter one with $n+n/2$ diametric 
pairs. 

\par By our remarks above, $\{x_{n+1},\dots, x_{2n}\}$ has diameter one and has $n/2$ diametric pairs.  Also note that
$$
|x_{n+j}-x_{2n+1}|^2=\frac{1}{4}+\left(\frac{\sqrt{3}}{2}-\sin\left(\frac{\pi}{2n}\right)\right)^2< 1
$$
for $j=1,\dots, n$. Moreover, 
\begin{align}
|x_{n+j}-x_{i}|^2&=\left|Tx_j-x_{i}+\sin\left(\frac{\pi}{2n}\right) e_3\right|^2\\
&=\left|Tx_j-x_{i}\right|^2+\sin\left(\frac{\pi}{2n}\right)^2\\
&=\frac{1}{4}+\frac{1}{4}-2Tx_j\cdot x_i+\sin\left(\frac{\pi}{2n}\right)^2\\
&=\frac{1}{2}-\frac{1}{2}\cos\left(\frac{\pi}{n}+\frac{\pi}{2n}j-\frac{\pi}{2n}i\right)+	\frac{1}{2}-\frac{1}{2}\cos\left(\frac{\pi}{n}\right)\\
&=1-\frac{1}{2}\left(\cos\left(\frac{\pi}{2n}(j-i+1/2)\right)+\cos\left(\frac{\pi}{n}\right)\right).
\end{align}
As we argued in the proof of Proposition \ref{EPpointsProp}, $|x_{n+j}-x_{i}|\le 1$ for all $i,j=1,\dots, n$, and 
for each $j=1,\dots, n$, there are exactly two indices $i\in \{1,\dots, n\}$ for which $|x_{n+j}-x_{i}|=1$.
Consequently, $\{x_1,\dots, x_{2n+1}\}$ has diameter one and $(n+n/2)+(n/2+2n)=4n$ diametric pairs. We conclude 
that $\{x_1,\dots,x_{2n+1}\}$ is extremal. 
\end{proof}
By exploiting the symmetry of $D_n$, it is possible to compute the perimeter and volume of this shape explicitly. However, the formulae are very long. Instead, we show some numerical approximations for the perimeter and volume in the table below. 
\begin{figure}[h]
\centering
   \includegraphics[width=.44\textwidth]{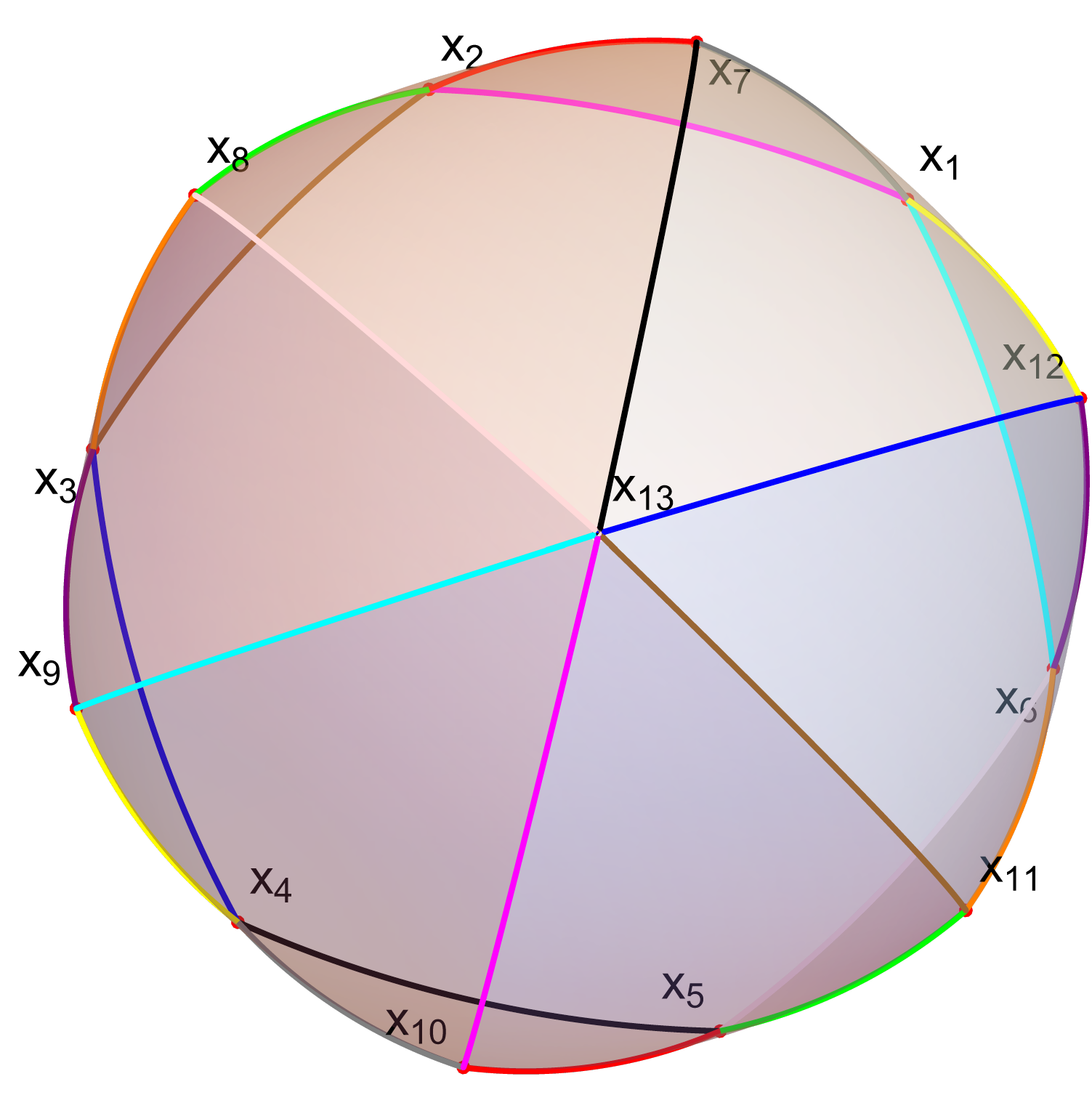} 
    \hspace{.3in}
      \includegraphics[width=.44\textwidth]{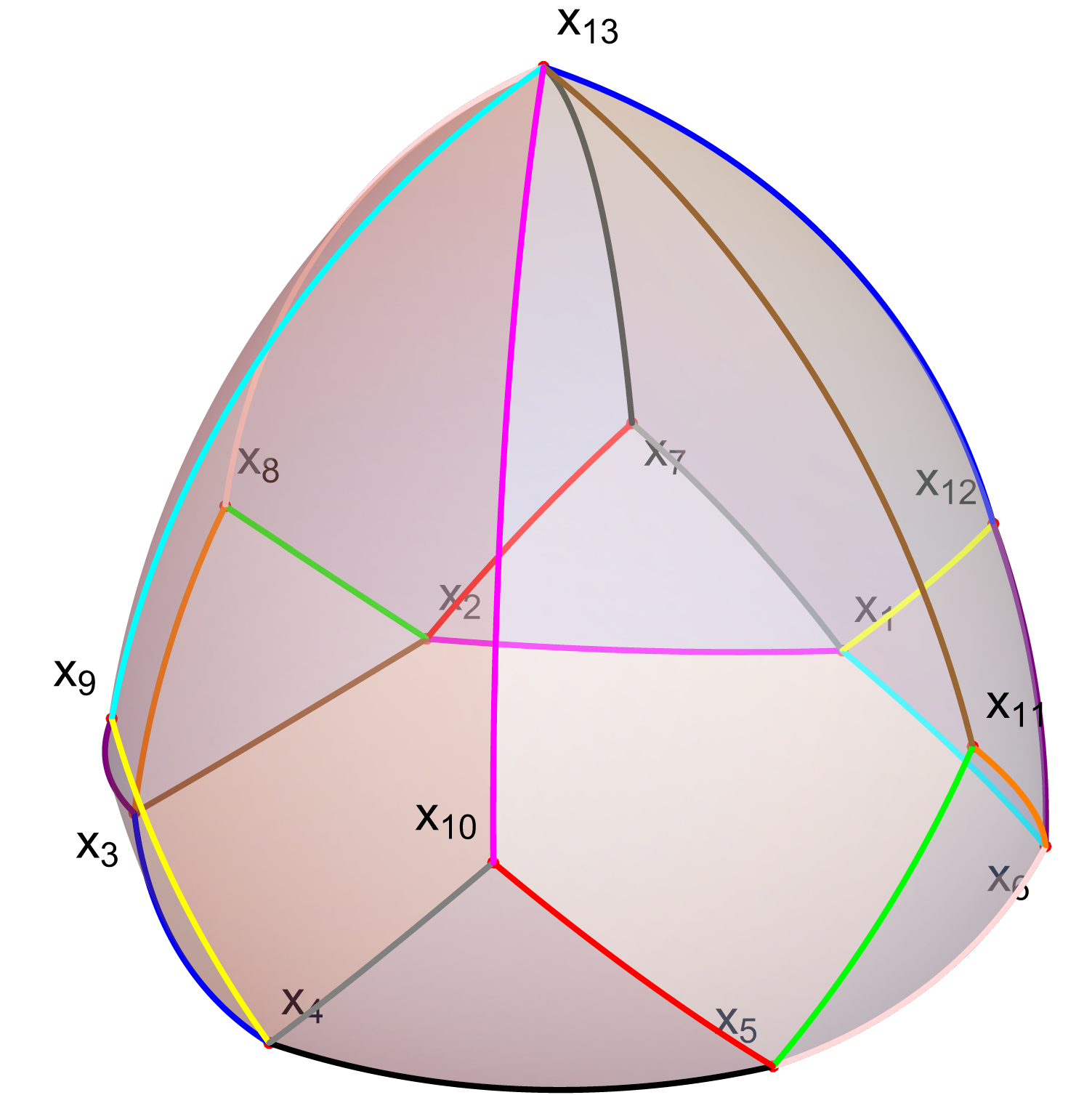} 
 \caption{Two views of the Reuleaux diminished trapezohedron $D_6$.}
\end{figure}
\begin{center}
\begin{tabular}{||c |l|} 
 \hline
$n$ & $P(D_n)$ \\
 \hline\hline
4 & 3.0260193893230074 \\
 \hline
6 & 3.017354597556886 \\
\hline
8 & 3.0097804740952094 \\
\hline 
10 & 3.0050348605116133 \\ 
\hline
12 & 3.0020158909800587 \\ 
\hline
14 & 3.0000079305708978 \\
\hline 
16 & 2.9986141631682397 \\ 
\hline
18& 2.9976106406122964 \\ 
 \hline
20 & 2.996865515987847 \\ 
 \hline
 22 & 2.9962977377798055\\
  \hline
 24 & 2.9958554883177815\\
   \hline
    26 & 2.9955044776712327\\
   \hline
\end{tabular}
\hspace{.5in}
\begin{tabular}{||c |l|} 
 \hline
$n$ & $V(D_n)$ \\
 \hline\hline
4 & 0.4633014137100808 \\
 \hline
 6& 0.45985656689727694\\
  \hline
8& 0.4565559505093948\\ 
 \hline
10 & 0.4544886558824768 \\
\hline
12 & 0.4531797978430361 \\
\hline 
14 & 0.4523131533481052 \\ 
\hline
16 & 0.45171383711833046 \\ 
\hline
18 & 0.451283641743075 \\
\hline 
20 & 0.4509650214589974 \\ 
\hline
22 & 0.4507227452212159 \\ 
 \hline
24 & 0.4505343676431973 \\ 
   \hline
   26 & 0.4503850797851724 \\ 
   \hline
\end{tabular} 
\end{center}

\section{Constant width}
A convex body $K\subset \R^3$ has {\it constant width} if the distance between any two parallel supporting planes for $K$ is equal to one. The simplest example of such a shape is a closed ball of radius one--half. However, there is an abundance of constant width shapes as discussed in the recent monograph \cite{MR3930585}.  We will consider a specific family below which can be designed from extremal sets.  In particular, we will show to compute their perimeters and volumes by combining Theorem \ref{PerThm} with a few other formulae discussed below. We will first need to recall the notion of a sliver and a spindle. 

\subsection{Slivers}
Suppose $X\subset \R^3$ is an extremal and $B(X)$ is the corresponding Reuleaux polyhedron with a dual edge pair $(e,e')$. Let us denote the endpoints of $e$ as $b,c\in X$ and the endpoints of $e'$ as $b',c'\in X$.  The geodesic $\gamma\subset \partial B(b')$ joining $b$ and $c$ also belongs to $\partial B(b')\cap B(X)$ by Proposition 4.2 of \cite{MR2593321} or Lemma 2.10 of \cite{HyndDensity}. The subset of the face $\partial B(b')\cap B(X)$ bounded by $\gamma$ and $e$ is the {\it sliver} in the face opposite $b'$ near the edge $e$.  See Figure \ref{mySliverFig}.  We will write $\text{Sl}(X,b',e)$ for this region. 
\begin{figure}[h]
\centering
   \includegraphics[width=.6\textwidth]{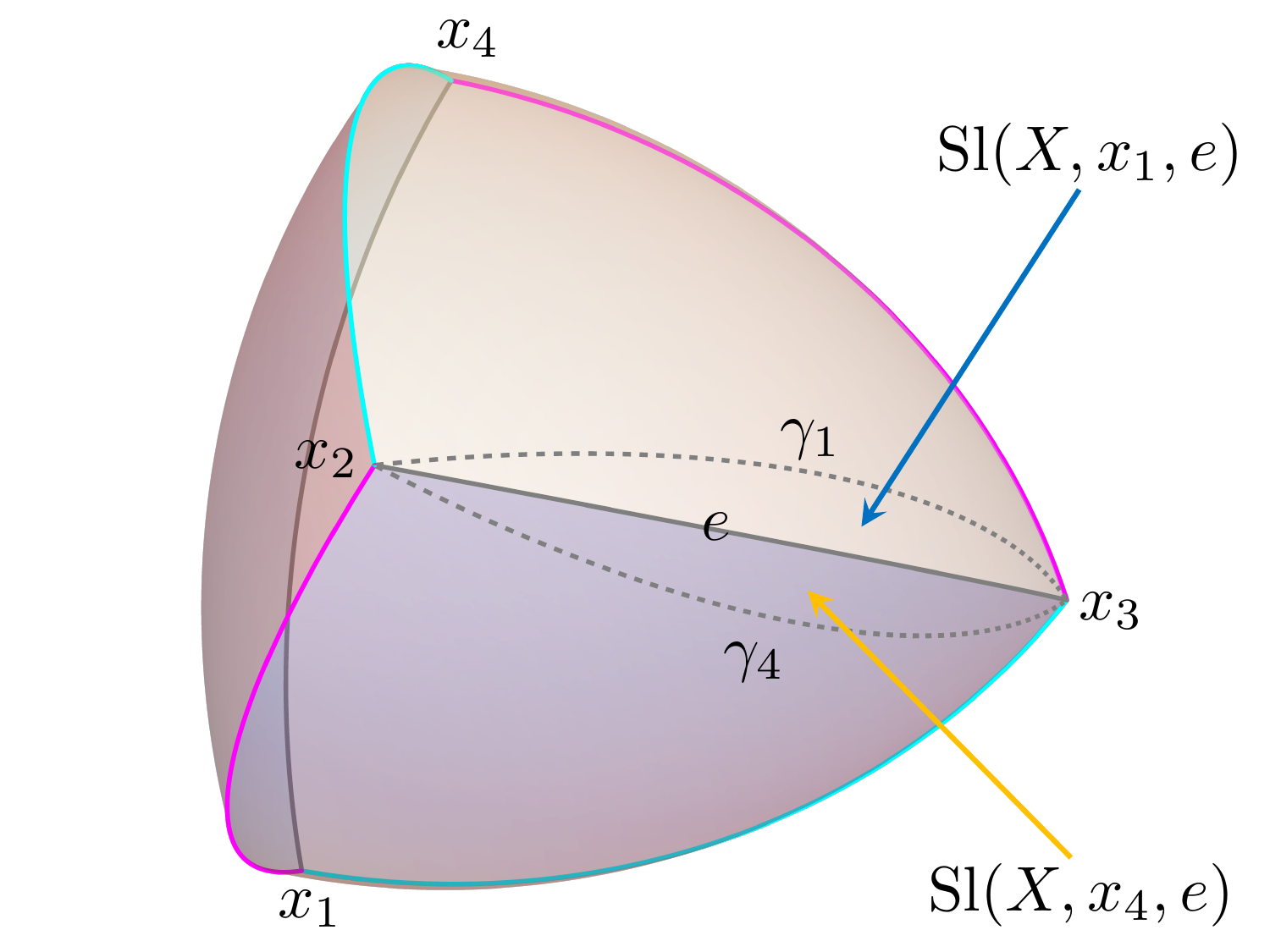} 
 \caption{A Reuleaux tetrahedron $B(X)$ with edge $e$ and geodesics $\gamma_1\subset \partial B(x_1)$ and
 $\gamma_4\subset \partial B(x_4)$ joining $x_2$ and $x_3$. The region of  the face $\partial B(x_1)\cap B(X)$ bounded by $\gamma_1$ and $e$ is the sliver in the face opposite $x_1$ near the edge $e$. We denote this sliver by $\text{Sl}(X,x_1,e)$. Likewise the region of  the face $\partial B(x_4)\cap B(X)$ bounded by $\gamma_4$ and $e$ is the sliver $\text{Sl}(X,x_4,e)$. }\label{mySliverFig}
\end{figure}
\par It will be useful for us to know the surface area of a sliver. Fortunately, we can use the Gauss--Bonnet formula to figure this out.
\begin{lem}
With the above notation,
\be\label{SliverArea}
\sigma(\textup{Sl}(X,b',e))=2\cos^{-1}\left(\sqrt{1-\left|\frac{b'-c'}{2}\right|^2}\frac{\sin\varphi}{\sin\phi}\right)-\left|\frac{c'-b'}{2}\right|\varphi.
\ee
Here 
$$
\varphi=\angle\left( b-(b'+c')/2,c-(b'+c')/2\right)
$$
and 
$$
\phi=\angle\left( b-b',c-b' \right).
$$
\end{lem}
\begin{proof}
We may assume without any loss of generality that 
\be
\begin{cases}
b=\sqrt{1-a^2}e_1\\
c=\sqrt{1-a^2}(\cos\varphi,\sin\varphi,0)\\
b'=-ae_3\\
c'=ae_3
\end{cases}
\ee
for $a=|b'-c'|/2$. Indeed, these coordinates can be obtained after an appropriate rotation and translation. Then 
the edge $e$ between $b$ and $c$ may be parametrized conveniently as 
$$
\eta(t)=\sqrt{1-a^2}(\cos t,\sin t,0)
$$
for $0\le t\le \varphi$. Similarly, 
we parametrize the geodesic $\gamma\subset \partial B(b')$ joining $b$ and $c$ via
\begin{align}
\zeta(t)&=b'+(b-b')\cos t+\frac{c-b'-\cos\phi (b-b')}{\sin\phi}\sin t\\
&=ae_3+(\sqrt{1-a^2}e_1+ae_3)\cos t+ \\
&\quad \quad \quad \frac{(\sqrt{1-a^2}(\cos\varphi,\sin\varphi,0)+ae_3)-\cos\phi (\sqrt{1-a^2}e_1+ae_3)}{\sin\phi}\sin t
\end{align}
for $0\le t\le \phi$. 

\par Applying the Gauss--Bonnet theorem as we did in our proof of Theorem \ref{PerThm} gives
\begin{align}
\sigma(\textup{Sl}(X,b',e))&=2\pi-\left(\int_{e}k_gds+\theta_e\right)-\left(\int_{\gamma}k_gds+\theta_\gamma\right)\\
&=(\pi-\theta_e)+(\pi-\theta_\gamma)-\int_{e}k_gds\\
&=(\pi-\theta_e)+(\pi-\theta_\gamma)-a\varphi.
\end{align}
Here we used that $k_g\equiv 0$ on $\gamma$. We also have
\begin{align}
\pi-\theta_\gamma&=\angle\left( \dot\eta(0) ,\dot\zeta(0) \right)\\
&=\angle\left(e_2,(\sqrt{1-a^2}(\cos\varphi,\sin\varphi,0)+ae_3)-\cos\phi (\sqrt{1-a^2}e_1+ae_3)\right)\\
&=\cos^{-1}\left(\sqrt{1-a^2}\frac{\sin\varphi}{\sin\phi}\right)
\end{align}
and
\begin{align}
\pi-\theta_e&=\angle\left(\dot\eta(\varphi), \dot\zeta(\phi) \right)\\
&=\angle\left((-\sin\varphi,\cos\varphi,0) ,(\sqrt{1-a^2}(\cos\varphi,\sin\varphi,0)+ae_3)\cos\phi -(\sqrt{1-a^2}e_1+ae_3)\right)\\
&=\cos^{-1}\left(\sqrt{1-a^2}\frac{\sin\varphi}{\sin\phi}\right).
\end{align}
As a result, 
\begin{align}
\sigma(\textup{Sl}(X,b',e))
&=2\cos^{-1}\left(\sqrt{1-a^2}\frac{\sin\varphi}{\sin\phi}\right)-a\varphi\\
&=2\cos^{-1}\left(\sqrt{1-\left|\frac{b'-c'}{2}\right|^2}\frac{\sin\varphi}{\sin\phi}\right)-\left|\frac{c'-b'}{2}\right|\varphi.
\end{align}
\end{proof}
\begin{cor}\label{SlivCor}
With the above notation, 
$$
\sigma(\textup{Sl}(X,b',e))=\sigma(\textup{Sl}(X,c',e)).
$$
\end{cor}
\begin{proof}
As $|b-b'|=|b-c'|=|c-b'|=|c-c'|=1,$
\begin{align}
|b-c|^2&=|b-b'-(c-b')|^2\\
&=|b-b'|^2+|c-b'|^2-2(b-b')\cdot (c-b')\\
&=2-2(b-b')\cdot (c-b')
\end{align}
and 
\begin{align}
|b-c|^2&=|b-c'-(c-c')|^2\\
&=|b-c'|^2+|c-c'|^2-2(b-c')\cdot (c-c')\\
&=2-2(b-c')\cdot (c-c').
\end{align}
It follows that $(b-b')\cdot (c-b')= (b-c')\cdot(c-c')$ and in turn
$$
\angle\left( b-b',c-b' \right)=\angle\left( b-c',c-c' \right).
$$
In view of \eqref{SliverArea}, the value of $\phi$ is the same for $\sigma(\textup{Sl}(X,b',e))$ and $\sigma(\textup{Sl}(X,c',e))$. We conclude $\sigma(\textup{Sl}(X,b',e))=\sigma(\textup{Sl}(X,c',e)).$
\end{proof}
\subsection{Spindles} 
We will once again consider an extremal $X\subset \R^3$ and a dual edge pair $(e,e')$ of the corresponding Reuleaux polyhedron $B(X)$. Let us write $b,c\in X$ for the endpoints of $e$ and $b',c'\in X$  for the endpoints of $e'$. There are two geodesic curves $\gamma_{b'}\subset \partial B(b')$ and $\gamma_{c'}\subset \partial B(c')$ which join $b$ and $c$ in their respective spheres. As mentioned above, $\gamma_{b'}\subset \partial B(b')\cap B(X)$ and $\gamma_{c'}\subset \partial B(c')\cap B(X)$.  Note specifically that $\gamma_{b'}$ can be rotated into $\gamma_{c'}$ by a rigid motion of $\R^3$ that fixes the line between $b$ and $c$.  

\par If we remove $\textup{Sl}(X,b',e)$ from $\partial B(b')\cap B(X)$ and $\textup{Sl}(X,c',e)$ from $\partial B(c')\cap B(X)$ this leaves a void in $\partial B(X)$. We define $\text{Sp}(X,e)$ as the surface obtained by rotating $\gamma_{b'}$ into $\gamma_{c'}$ which fills this void.  The shape $\text{Sp}(X,e)$ is sometimes called a {\it spindle surface} as it is a portion of spindle torus.  We compute the surface area of this spindle portion below.  
\begin{figure}[h]
\centering
      \includegraphics[width=.6\textwidth]{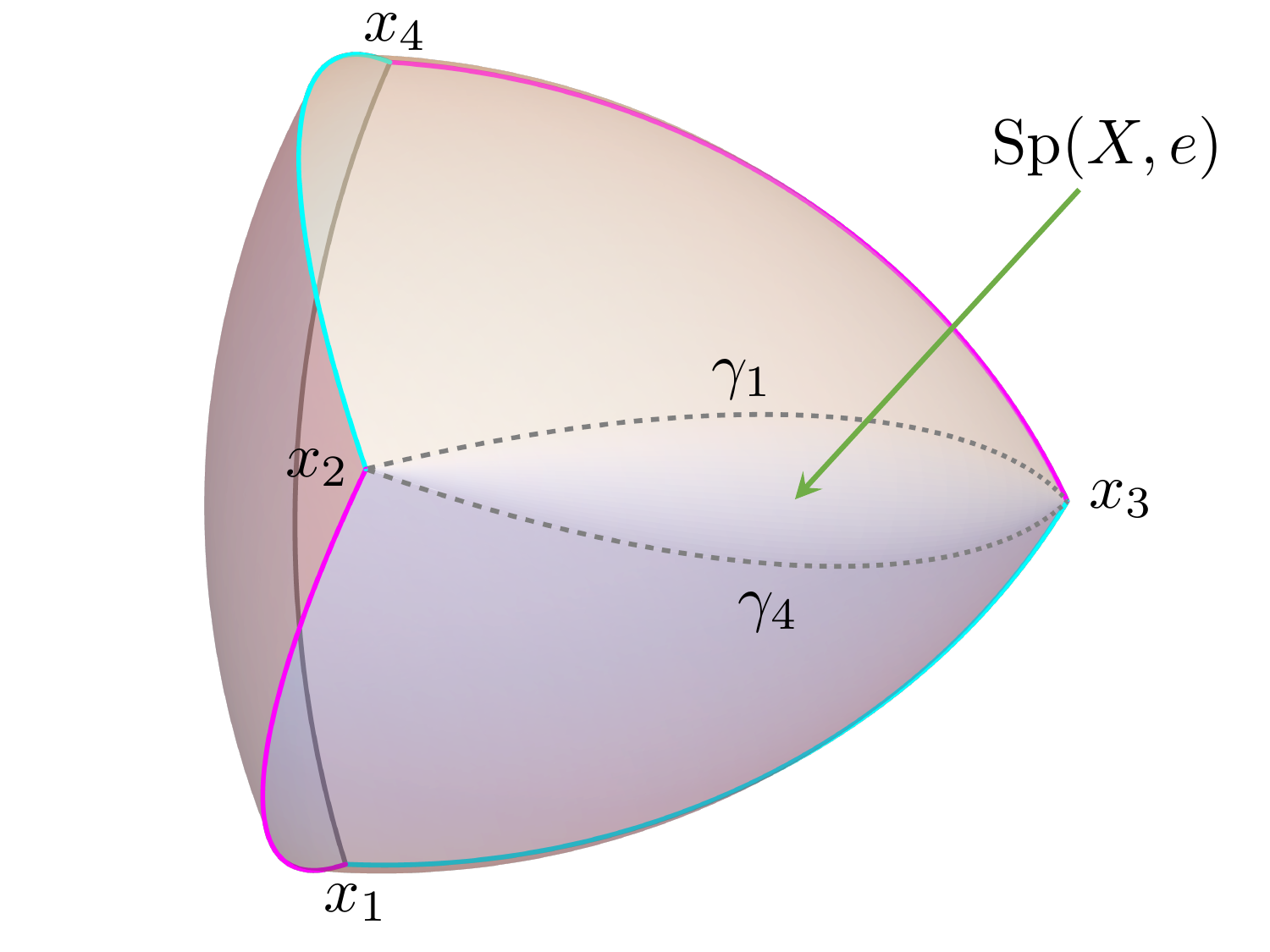} 
 \caption{This is the boundary of a Reuleaux tetrahedron $B(X)$ with the slivers $\textup{Sl}(X,x_1,e)$ and $\textup{Sl}(X,x_4,e)$ replaced by the spindle surface $\text{Sp}(X,e)$; here $e$ is the edge of $B(X)$ which joins $x_2$ and $x_3$.  Compare with Figure \ref{mySliverFig}. Note in particular that $\text{Sp}(X,e)$ is the surface obtained by rotating the geodesic $\gamma_1$ which joins $x_2$ and $x_3$ in $\partial B(x_1)$ into the geodesic $\gamma_4$ which joins $x_2$ and $x_3$ in $\partial B(x_4)$ about the line which passes through $x_2$ and $x_3$.}\label{Surgery}
 \end{figure}

\begin{lem}\label{SpindleLemma}
With the above notation,
$$
\sigma(\textup{Sp}(X,e))=2\psi \left(-\sqrt{1-\left|(b-c)/2\right|^2}\sin^{-1}\left(|b-c|/2\right)+|b-c|/2\right).
$$
Here $\psi=\angle\left( b'-(b+c)/2,c'-(b+c)/2\right)$.
\end{lem}
\begin{proof}
Without any loss of generality, we may assume 
\be
\begin{cases}
b=ae_3\\
c=-ae_3\\
b'=\sqrt{1-a^2}(-\cos\psi,-\sin\psi,0)\\
c'=-\sqrt{1-a^2}e_1
\end{cases}
\ee
for $a=|b-c|/2$. In this case, $\textup{Sp}(X,e)$ is obtained by rotating the arc 
$$
(x_1+\sqrt{1-a^2})^2+x_3^2=1,\quad  x_1\ge 0
$$
in the $x_1x_3$--plane counterclockwise $\psi$ units about the $x_3$--axis.  Direct computation of this surface area yields
\begin{align}
\sigma(\textup{Sp}(X,e))&=2\psi \left(-\sqrt{1-a^2}\sin^{-1}(a)+a\right)\\
&=2\psi \left(-\sqrt{1-\left|(b-c)/2\right|^2}\sin^{-1}\left(|b-c|/2\right)+|b-c|/2\right).
\end{align}
\end{proof}
\subsection{Meissner polyhedra}
Suppose that $X$ is extremal, $(e,e')$ is a dual edge pair of $B(X)$, $b,c\in X$ are the endpoints of $e$, and $b',c'\in X$ are the endpoints of $e'$.  We will {\it smooth} edge $e$ of $ B(X)$ by replacing the slivers $\text{Sl}(X,b',e)$ and $\text{Sl}(X,c',e)$ with the spindle surface $\textup{Sp}(X,e)$; see Figure \ref{Surgery}.  If we smooth one edge in each dual edge pair of $B(X)$, we obtain a surface which bounds a constant width body. This was verified by Montejano and Roldan-Pensado in \cite{MR3620844}; see also section 4 of \cite{HyndDensity}. Such a shape is known as a {\it Meissner polyhedron} based on $X$.
\begin{figure}[h]
\centering
   \includegraphics[width=.45\textwidth]{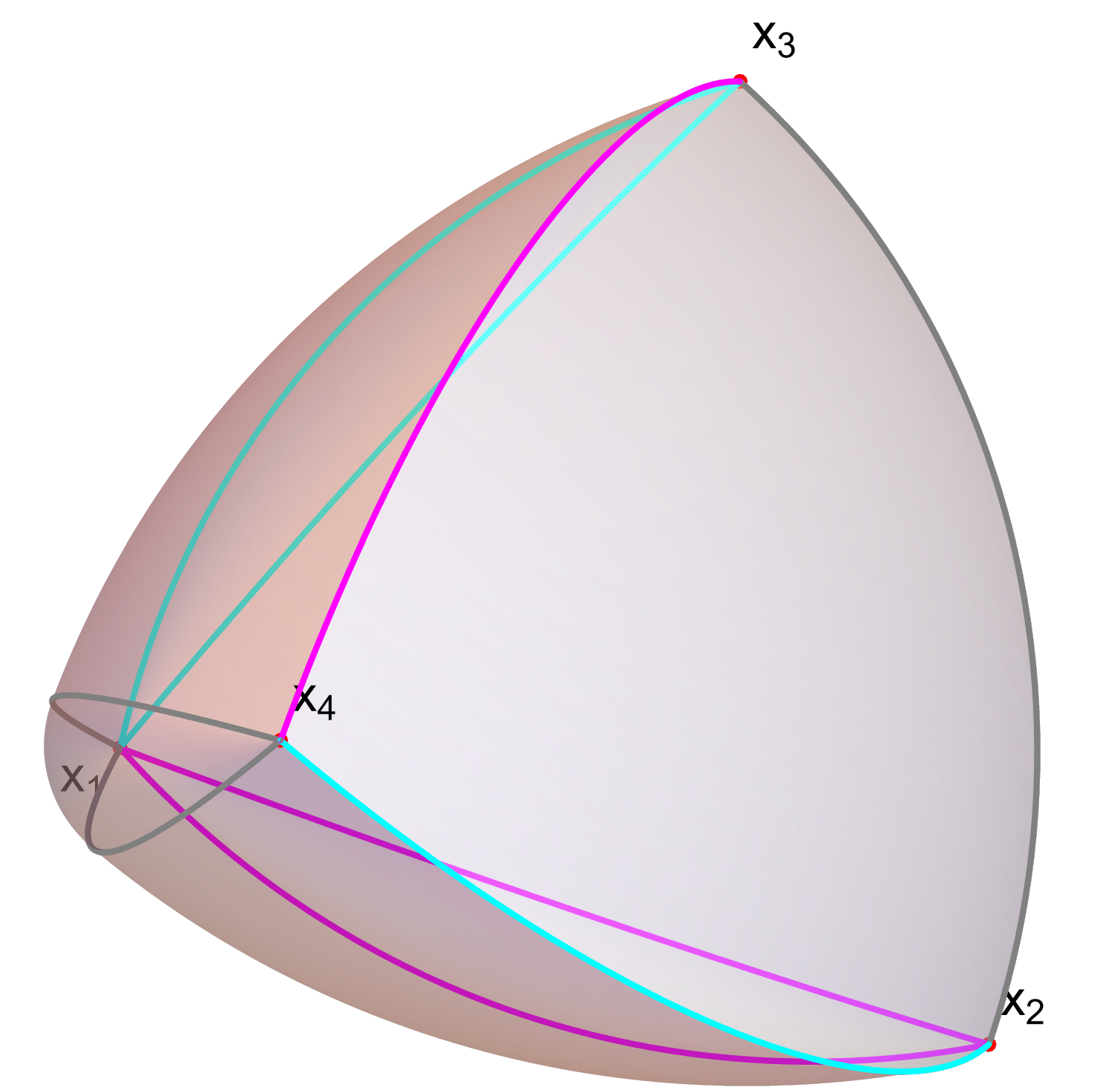} 
    \hspace{.3in}
      \includegraphics[width=.44\textwidth]{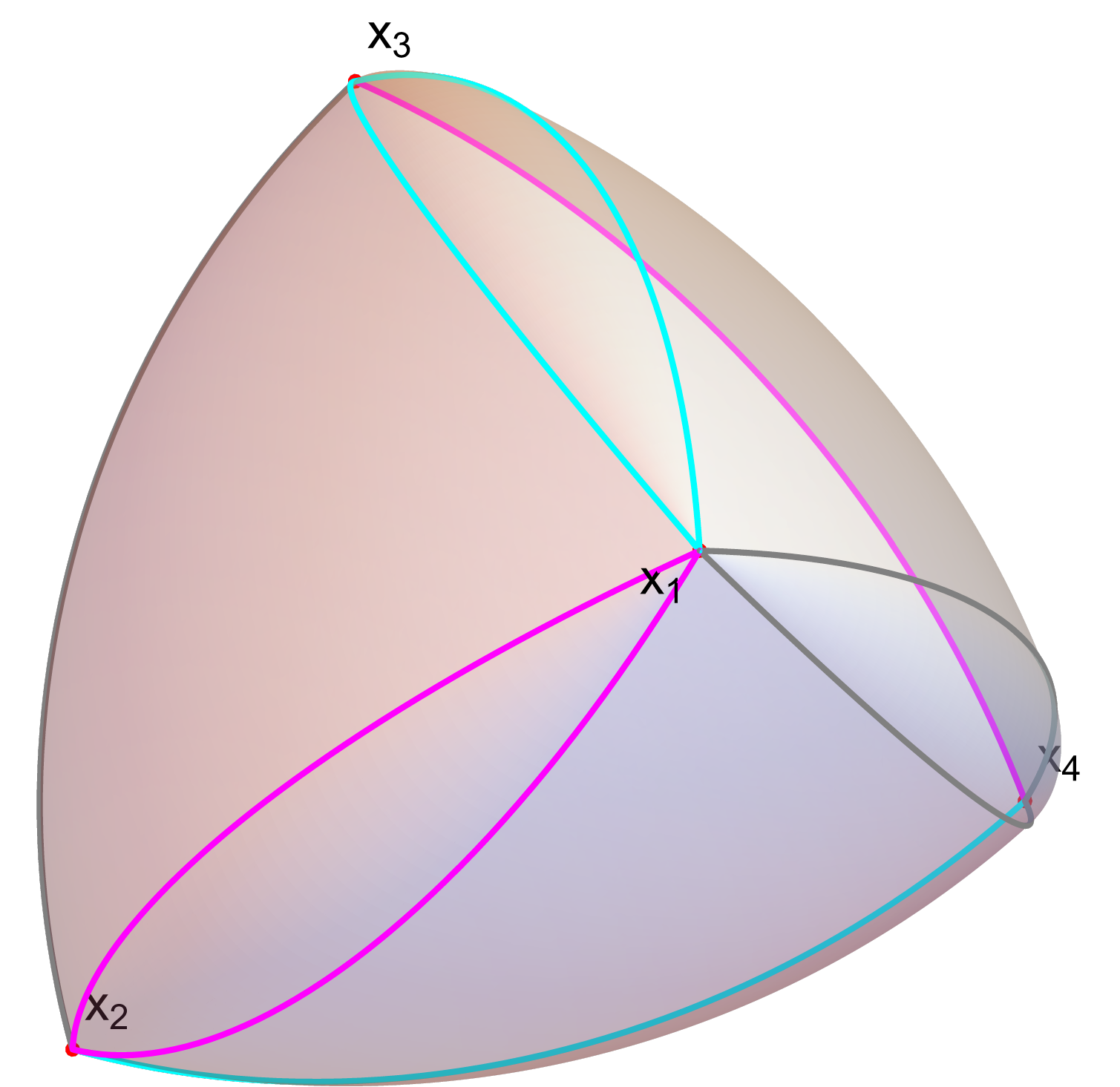} 
 \caption{
 A Meissner tetrahedron such that edges which join $x_1$ and $x_2$, $x_1$ and $x_3$, and $x_1$ and $x_4$ have been smoothed. Note that instead of indicating dual edge pairs with the same color, we have indicated the geodesics which bound the spindle surface with the same color of the edge it smooths.}\label{MessnerTetra1}
 \end{figure}
\par The simplest Meissner polyhedra are the two Meissner tetrahedra. These are Meissner polyhedra based on the vertices of a regular tetrahedron in $\R^3$ of side length one. The two types of shapes one obtains are displayed in Figures \ref{MessnerTetra1} and \ref{MessnerTetra2}. 
 \begin{figure}[h]
\centering
   \includegraphics[width=.44\textwidth]{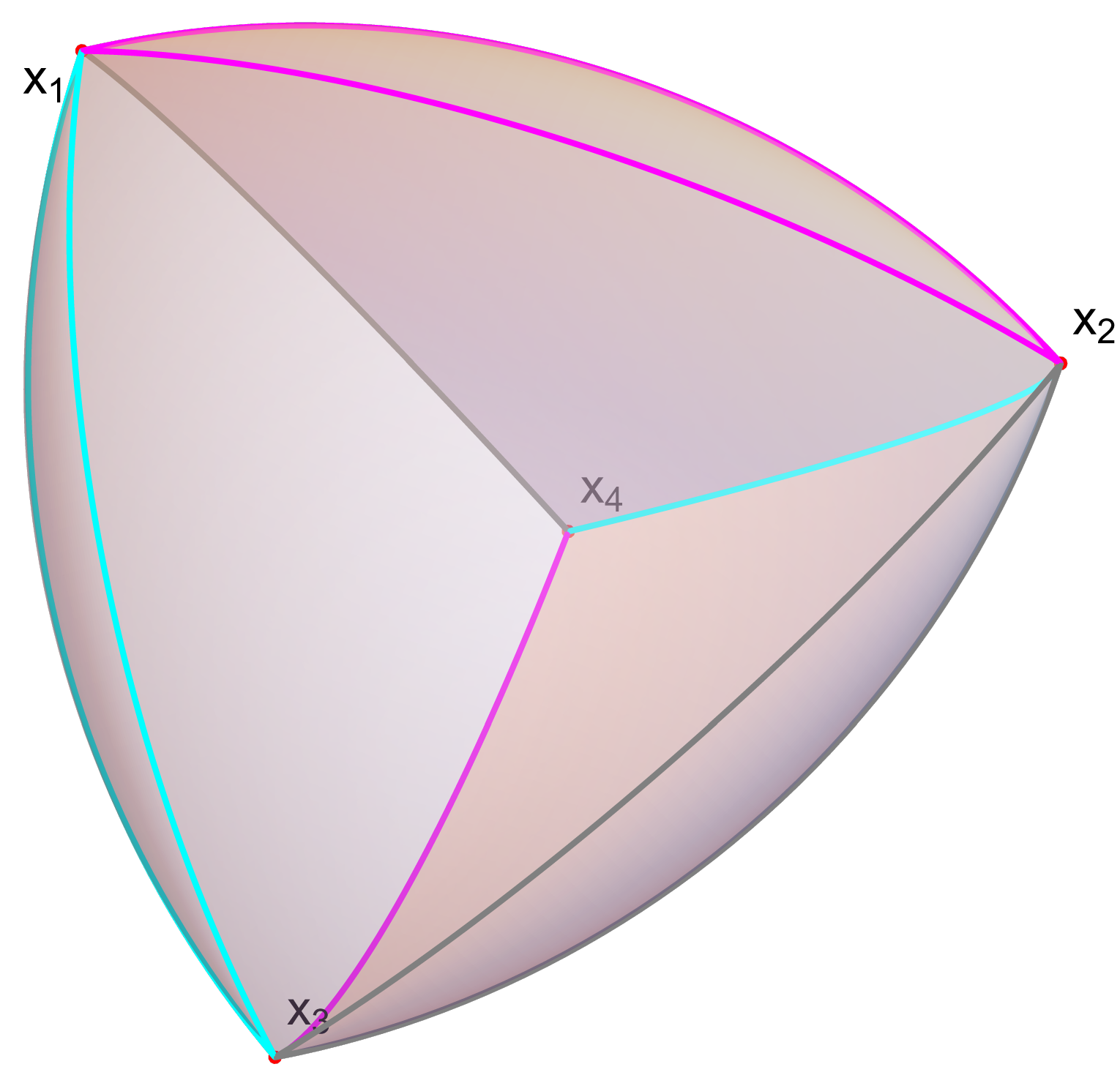} 
    \hspace{.3in}
      \includegraphics[width=.44\textwidth]{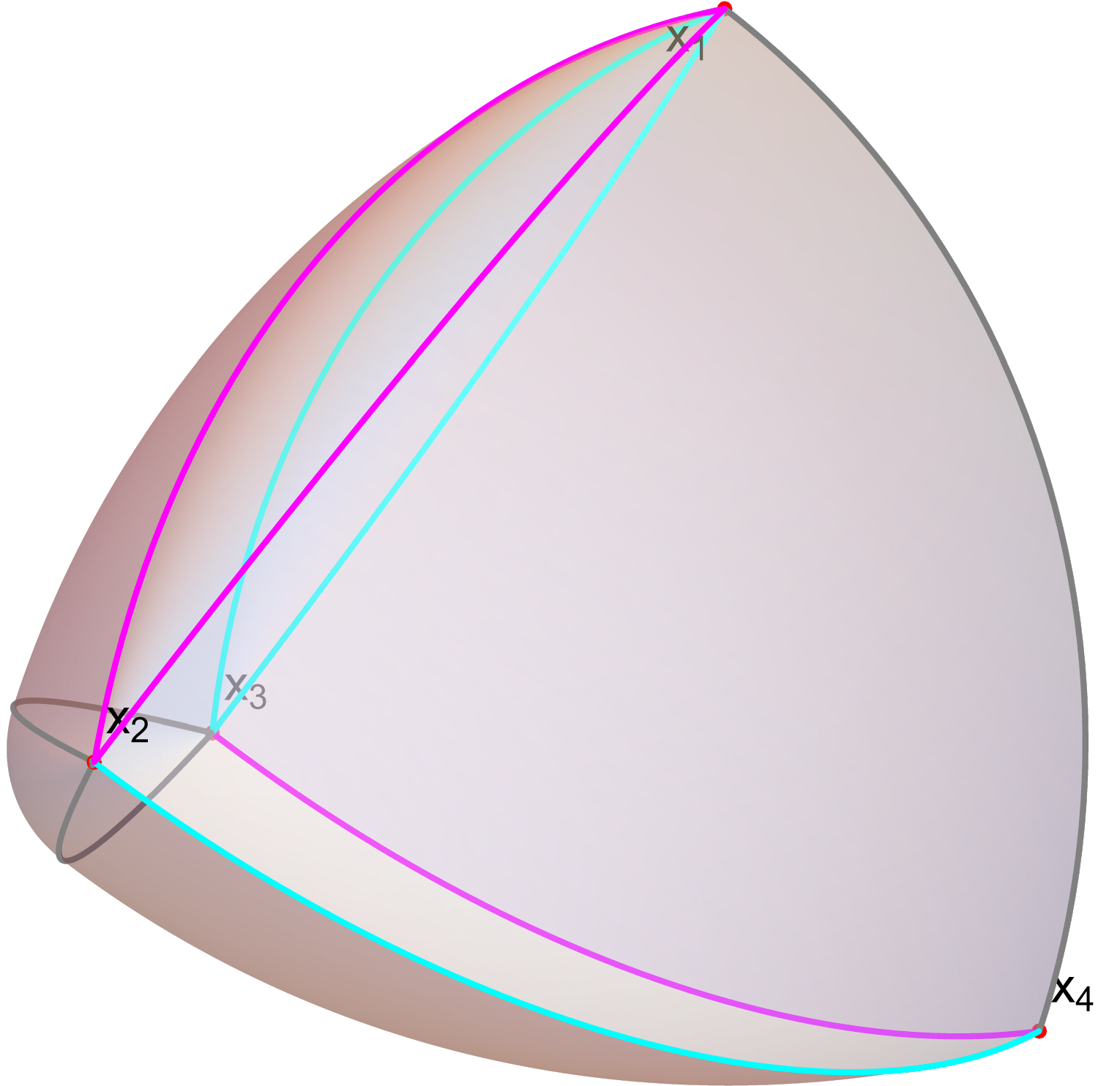} 
 \caption{A Meissner tetrahedron in which three smoothed edges share a common face. }\label{MessnerTetra2}
 \end{figure}
Meissner polyhedra are of special interest as they are known to be dense among all constant width shape \cite{HyndDensity}.  The two Meissner tetrahedra are also conjectured to enclose least volume among all constant width shapes \cite{MR2844102, MR1316393}.

\par Employing our computations above, we can approximate the perimeter of any Meissner polyhedron $M\subset \R^3$. Indeed, if $B(X)$ is a Reuleaux polyhedron with dual edge pairs 
$$
(e_1,e_1'),\dots, (e_{m-1},e_{m-1}'),
$$
and $M$ is the Meissner polyhedron obtained by smoothing edges $e_1,\dots, e_{m-1}$, then
\be\label{SurgeryFormula}
P(M)=P(B(X))+\sum^{m-1}_{j=1}\left[\sigma(\textup{Sp}(X,e_j))-
2\sigma(\textup{Sl}(X,b_j',e_j))\right].
\ee
Here $b_j',c_j'\in X$ are the endpoints of $e_j'$ and we used $\sigma(\textup{Sl}(X,b_j',e_j))=\sigma(\textup{Sl}(X,c_j',e_j))$, which was verified in Corollary \ref{SlivCor}. Rapid computation of this perimeter is possible 
by implementing the formulae we derived in Theorem \ref{PerThm}, Lemma \ref{SliverArea}, and  Lemma \ref{SpindleLemma}. We also emphasize that this computation has been taken further by Bogosel who managed to find a remarkable formula for $P(M)$ in terms of the dual edge pairs and distance between the endpoints of the edges of $B(X)$ \cite{bogosel2023volume}. 

\par Finally, we can use Blaschke's relation
$$
V(M)=\frac{1}{2}P(M)-\frac{\pi}{3}
$$
(Theorem 12.1.4 of \cite{MR3930585}) to approximate the volume of $M$. Indeed, once we have an  approximation for the perimeter of $M$, we also have one for the volume of $M$.  
 We will conclude our study by approximating the perimeter and volume of a few examples of Meissner polyhedra. 

\begin{ex} Suppose $M$ is a Meissner tetrahedron based on 
the vertices of a regular tetrahedron $X=\{x_1,x_2,x_3,x_4\}$.  Let $e$ denote the edge between vertices $x_1$ and $x_2$ and $e'$ the edge which joins $x_3$ and $x_4$; note that $e$ and $e'$ are dual edges of $B(X)$. It is well known the dihedral angle of a regular tetrahedron is $\cos^{-1}\left(1/3\right)$; this is verified in the appendix of \cite{HyndDensity} for instance. This observation implies  
$$
\psi=\angle\left(x_1-(x_3+x_4)/2,x_2-(x_3+x_4)/2\right)=\cos^{-1}\left(1/3\right)
$$
and 
$$
\varphi=\angle\left(x_3-(x_1+x_2)/2,x_4-(x_1+x_2)/2\right)=\cos^{-1}\left(1/3\right).
$$
By Lemma \ref{SpindleLemma},
\begin{align}
\sigma(\textup{Sp}(X,e))&=2\cos^{-1}\left(1/3\right) \left(-\frac{\sqrt{3}}{2}\sin^{-1}\left(\frac{1}{2}\right)+\frac{1}{2}\right)\\
&=\cos^{-1}\left(1/3\right) \left(-\frac{\pi}{2\sqrt{3}}+1\right).
\end{align}

\par Let us now compute the area of the sliver near edge $e$ in the face opposite $x_4$. As $x_1,x_2,x_4$ belong to an equilateral triangle, $\phi=\angle\left(x_1-x_4,x_2-x_4 \right)=\pi/3.$
Thus, 
\begin{align}
\sigma(\textup{Sl}(X,x_4,e))&=2\cos^{-1}\left(\frac{\sqrt{3}}{2}\frac{\sin\varphi}{\displaystyle\frac{\sqrt{3}}{2}}\right)-\frac{1}{2}\cos^{-1}\left(\frac{1}{3}\right)\\
&=2\cos^{-1}\left(\sin\varphi\right)-\frac{1}{2}\cos^{-1}\left(\frac{1}{3}\right)\\ 
&=\pi-2\sin^{-1}(\sin\varphi)-\frac{1}{2}\cos^{-1}\left(\frac{1}{3}\right)\\
&=\pi-\frac{5}{2}\cos^{-1}\left(\frac{1}{3}\right).
\end{align}
By symmetry, this is the same for any sliver on the boundary of $B(X)$.

\par In order view of \eqref{SurgeryFormula}, the perimeter of $M$ is 
\begin{align}
P( M)&=P(B(X))+3\cos^{-1}\left(\frac{1}{3}\right)\left(-\frac{\pi}{2\sqrt{3}}+1\right)-6\left(\pi-\frac{5}{2}\cos^{-1}\left(\frac{1}{3}\right)\right)\\
&=4\left(2\pi -\frac{9}{2}\cos^{-1}\left(\frac{1}{3}\right)\right)-6\pi +3\cos^{-1}\left(\frac{1}{3}\right)\left(6-\frac{\pi}{2\sqrt{3}}\right)\\
&=\pi\left(2-\frac{\sqrt{3}}{2}\cos^{-1}(1/3)\right).
\end{align}
Blaschke's relation also gives 
$$
V(M)=\frac{1}{2}P(M)-\frac{\pi}{3}=\pi\left(\frac{2}{3}-\frac{\sqrt{3}}{4}\cos^{-1}(1/3)\right).
$$
Lastly, we emphasize that these computations are valid for both types of Meissner tetrahedra. 
\end{ex}
 \begin{figure}[h]
\centering
   \includegraphics[width=.44\textwidth]{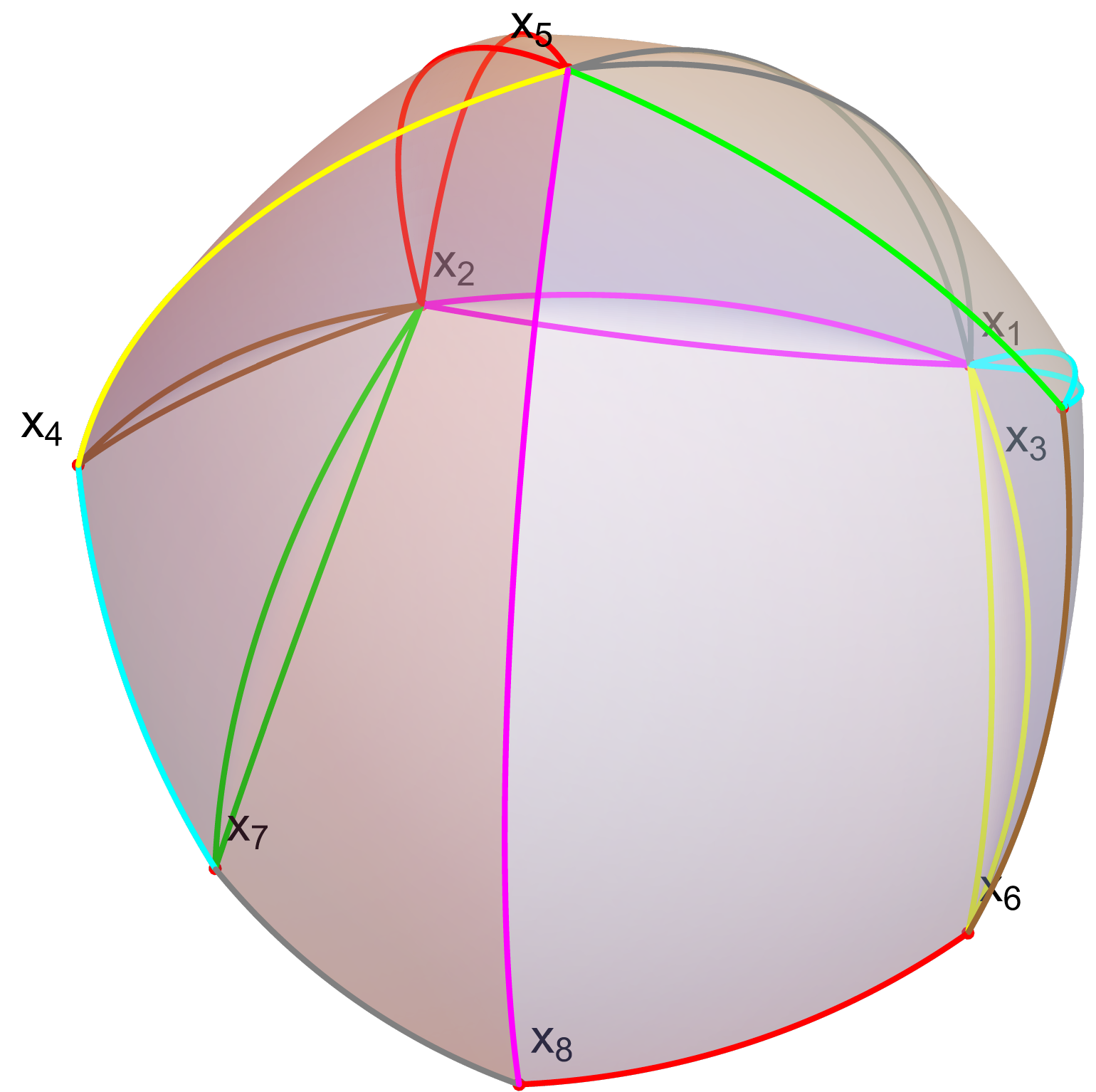} 
    \hspace{.3in}
      \includegraphics[width=.44\textwidth]{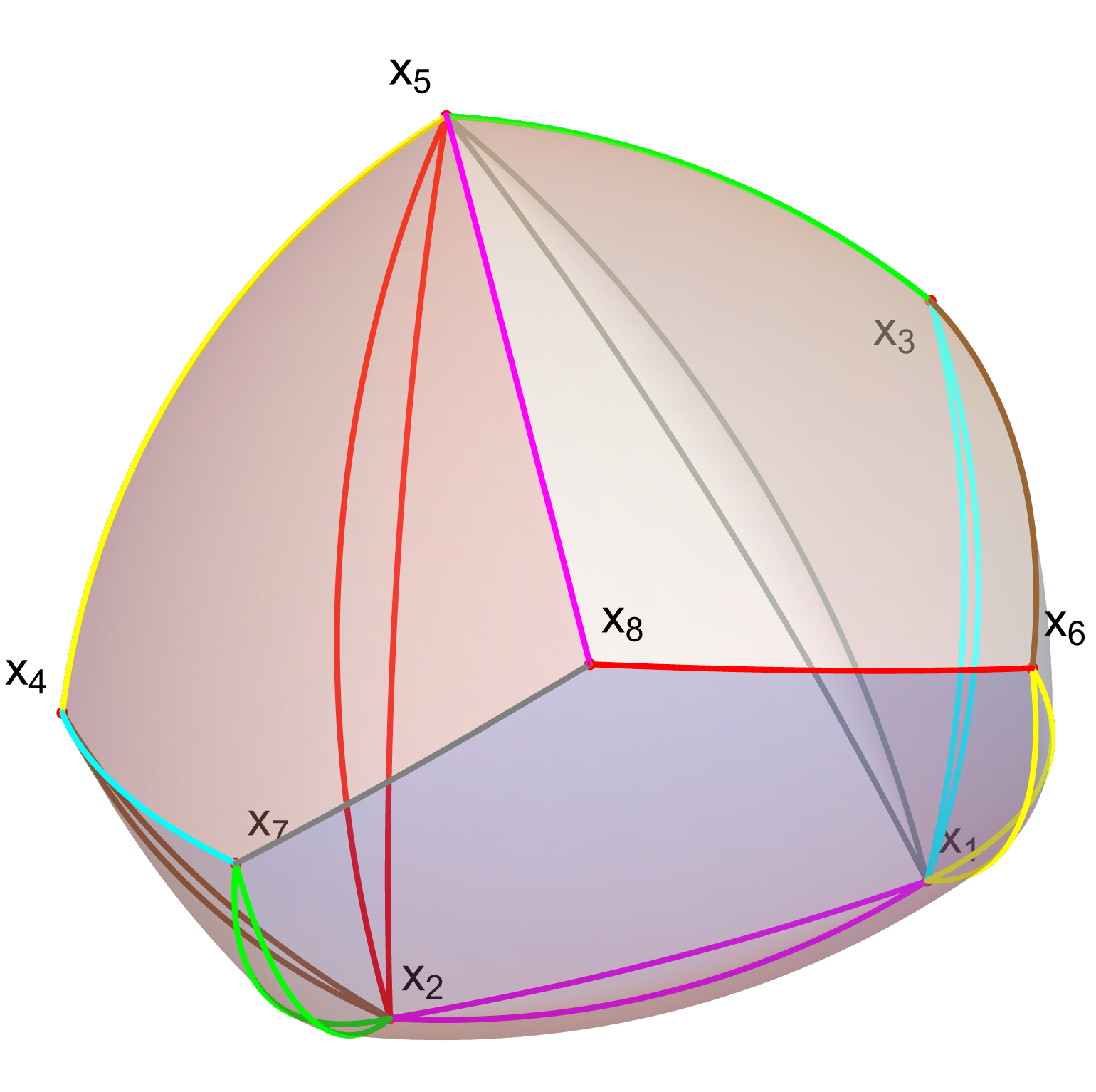} 
 \caption{A Meissner polyhedron obtained from smoothing edges of $B(X)$, where $X=\{x_1,\dots, x_8\}$ is given in  Example \ref{BullExample}.}\label{MessnerBull}
 \end{figure}
\begin{ex}
Consider the extremal set $X=\{x_1,\dots, x_8\}$ presented in Example \ref{BullExample} above.  This set has $14$ edges and $7$ dual edge pairs. If we denote the edge $(i,j)$ as the one which joins vertices $x_i$ and $x_j$, then the dual edge pairs may be enumerated as
$$
\begin{cases}
(1,2), (5,8)\\
(1,3), (4,7)\\
(1,5), (7,8)\\
(1,6), (4,5)\\
(2,4), (3,6)\\
(2,5), (6,8)\\
(2,7), (3,5).
\end{cases}
 $$
 The Meissner polyhedra $M$ obtained from smoothing the boundary of $B(X)$ near the edges $(1,2),(1,3), (1,5), (1,6), (2,4), (2,5), (2,7)$ is displayed in Figure \ref{MessnerBull}. The perimeter and volume given by \eqref{SurgeryFormula} and Blaschke's relation are approximately
 $$
 P(M)\approx 2.9968929812165475\; \text{ and }\; V(M)\approx 0.4512489394116761.
 $$
\end{ex}
 \begin{figure}[h]
\centering
   \includegraphics[width=.45\textwidth]{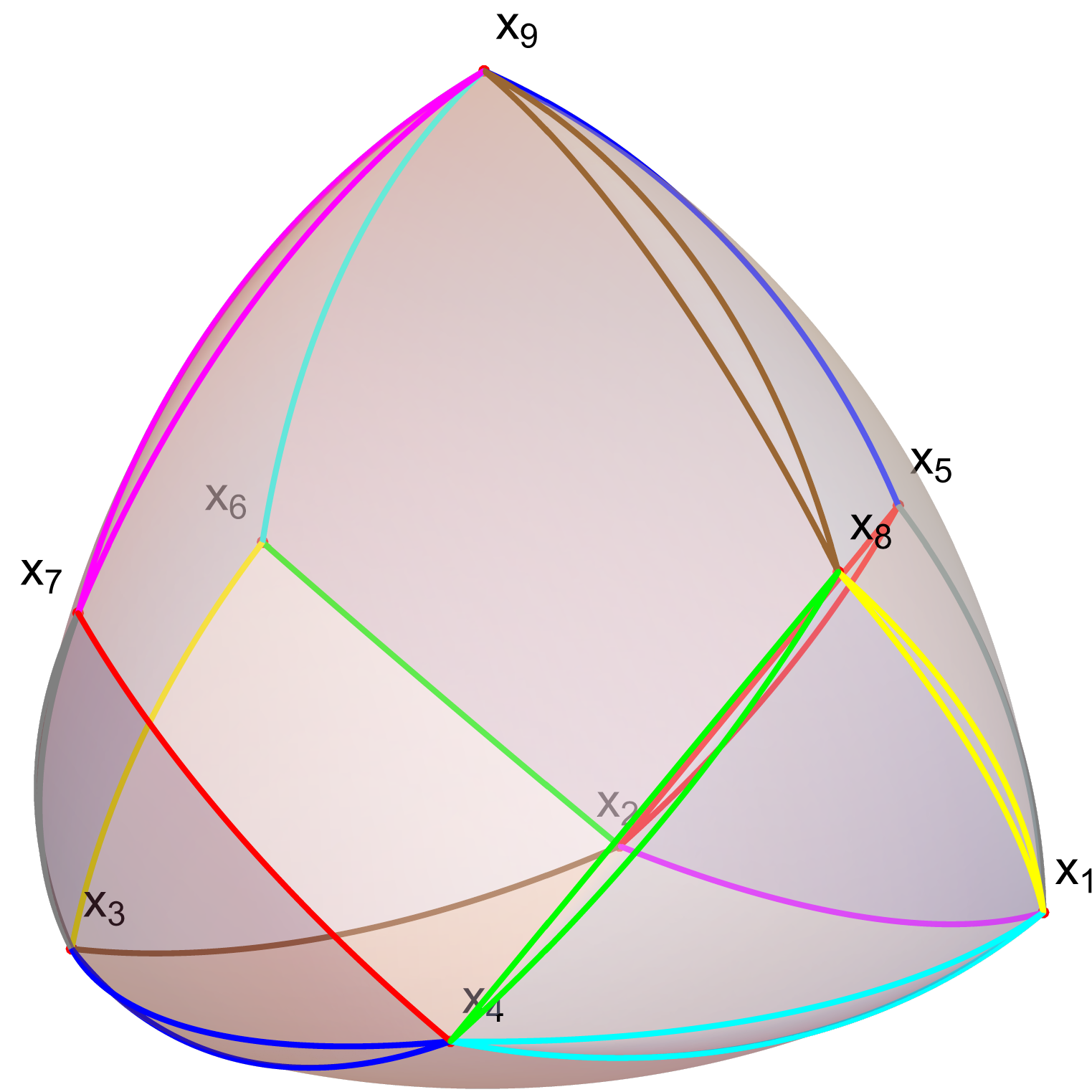} 
    \hspace{.3in}
      \includegraphics[width=.44\textwidth]{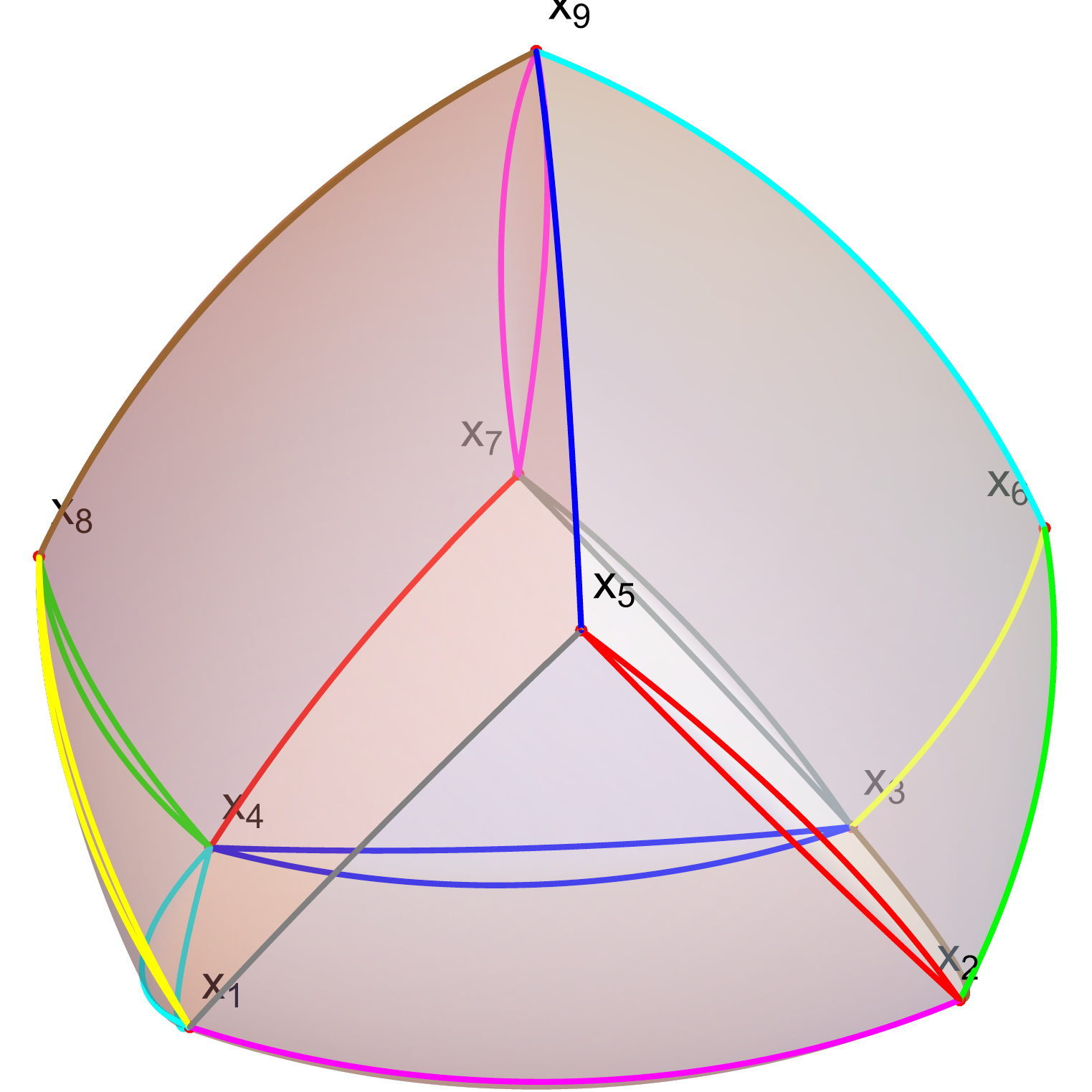} 
 \caption{A Meissner polyhedron obtained from smoothing edges of the diminished Reuleaux trapezohedron $D_4$.}\label{MessnerDimTrap}
 \end{figure}
\begin{ex}
Recall the diminished Reuleaux trapezohedron $D_4=B(\{x_1,\dots,x_9\})$ discussed in Section \ref{DiminishedTrapSect} and displayed in Figure \ref{DimTrapFig}.  The indices of the dual edge pairs are 
$$
\begin{cases}
(1,2), (7,9)\\
(1,4), (6,9)\\
(1,5),(3,7)\\
(1,8), (3,6)\\
(2,3),(8,9)\\
(2,5),(4,7)\\
(2,6), (4,8)\\
(3,4), (5,9).
\end{cases}
$$
We consider the Meissner polyhedron $M$ obtained by smoothing the edges $(7,9), (1,4), (3,7)$, $(1,8), (8,9), (2,5), (4,8), (3,4)$. This shape is displayed in Figure \ref{MessnerDimTrap}, and its perimeter and volume are approximately
$$
P(M)\approx 3.0207439602029265 \text{ and } V(M)\approx 0.4631744289048656.
$$
\end{ex}
\bibliography{PerVolBib}{}

\bibliographystyle{plain}

\typeout{get arXiv to do 4 passes: Label(s) may have changed. Rerun}

\end{document}